\documentclass[12pt]{amsart}
\usepackage[left=2.5cm,right=2.5cm, top=2.5cm, bottom=2.5cm]{geometry}
\usepackage{amsmath,amsfonts,amssymb}
\usepackage{color} 
\usepackage{enumerate,graphics, graphicx}
\usepackage{enumitem}
\usepackage{tikz}
\usepackage[all,cmtip]{xy}
	\usetikzlibrary{arrows}
\usepackage[pdftex,bookmarks=true]{hyperref}
\usepackage{verbatim}
\usepackage{caption}
\usepackage{subcaption}

\newtheorem{theorem}{Theorem}[section]
\newtheorem{lemma}[theorem]{Lemma}
\newtheorem{conjecture}[theorem]{Conjecture}
\newtheorem{proposition}[theorem]{Proposition}
\newtheorem{prop}[theorem]{Proposition}
\newtheorem{corollary}[theorem]{Corollary}

\theoremstyle{definition}
\newtheorem{defn}[theorem]{Definition}

\newtheorem{eg}[theorem]{Example}

\newtheorem{question}[theorem]{Question}

\newtheorem{rmk}[theorem]{Remark}

\newcommand{\R}{\mathbb R}

\def\been{\begin{enumerate}}
\def\enen{\end{enumerate}}

\usepackage{fancybox}
\usepackage{todonotes}
\usepackage{menukeys}
\newcounter{todocounter}

\presetkeys{todonotes}{inline, color=blue!30}{}

\title[Identifiability of linear compartmental models]{Linear compartmental models: input-output equations and operations that preserve identifiability}
\author[Gross]{Elizabeth Gross}
\address{University of Hawaii at Manoa}
\author[Harrington]{Heather Harrington}
\address{University of Oxford}
\author[Meshkat]{Nicolette Meshkat}
\address{Santa Clara University}
\author[Shiu]{Anne Shiu}
\address{Texas A\&M University}

\date{May 24, 2019}

\begin{document}
\maketitle


\begin{abstract}
This work focuses on the question of how identifiability of a mathematical model, that is, whether parameters can be recovered from data, is related to identifiability of its submodels.  We look specifically at linear compartmental models and investigate when identifiability is preserved after adding or removing model components. 
In particular, we examine whether identifiability is preserved when an input, output, edge, or leak is added or deleted.  
Our approach, via differential algebra, is to analyze specific input-output equations of a model and the Jacobian of the associated coefficient map.
We clarify a prior determinantal formula for these equations, 
and then use it to 
prove that, under some hypotheses, a model's 
input-output equations can be understood in terms of 
certain submodels we call ``output-reachable''.
Our proofs use algebraic and combinatorial techniques.


\end{abstract}

\maketitle

\section{Introduction} \label{sec:intro} Identifiability refers to the
property possessed by a mathematical model when the model's
  parameters can be recovered from data.  We focus on structural
identifiability, that is, whether the model equations allow
unique determination of a finite number of parameters
from noise-free and continuous data, henceforth
referred to as identifiability. 
A model is identifiable if the map from
  the parameters of the model to output trajectories is injective.
  We assume the model is known, so our focus is on parameter
  identifiability rather than model identifiability.

Here we investigate the problem of assessing
identifiability for a class of models used extensively in biological
applications, namely, linear compartmental
models~\cite{distefano-book,godfrey}.  Our interest is in the
following question: When is identifiability preserved after components
of the model -- such as inputs, outputs, leaks, or edges -- are added
or removed?  For example, in a prior work~\cite{singularlocus},
sufficient conditions are given for when edges can be removed.  In
this work, we show that adding outputs or inputs -- or, under certain
hypotheses, adding or removing a leak -- preserves
identifiability. These results, as well as additional new
contributions, are summarized in Tables~\ref{table:add-delete}
and~\ref{table:add-delete-uniden} below.

The question of when a submodel of an identifiable model is
identifiable has both theoretical significance and real-life
applications (e.g., identifiability of nested models in epidemiology
\cite{Tuncer2016}). On the theory side, the investigation of submodels
of identifiable models was initiated by Vajda and others
\cite{vajda1984, vajdaetal} with the goal of reducing the 
 problem of assessing identifiability to simpler and more tractable computations. Since then, it has remained an interesting problem, as it addresses one of the most important questions regarding linear compartmental models: \emph{Which models are identifiable?}  
  
In applications, deleting an input or output corresponds to changing the experimental setup or changing the design of a biological circuit within a cell. Hence, if we know that the corresponding submodel remains identifiable, the new setup 
will not affect the ability to recover parameters. As for edges and leaks, deleting these elements is a way to model a biological intervention or knockout.  Also, in practice,
the precise model architecture may be unknown, so 
being able to investigate identifiability of a few candidate models at once is desirable.

Another motivation for our work is its application to chemical reaction networks.  Linear compartmental models correspond to monomolecular chemical reaction networks, and we will use the results here to analyze, in subsequent work~\cite{join}, 
when identifiability is preserved when networks are joined or decomposed.

For linear compartmental models, the problem of assessing (generic local) identifiability 
can be translated, through standard differential algebra techniques, 
to the question of whether the Jacobian matrix of the {\em coefficient map} (arising from certain {\em input-output equations}) is generically full rank.
A formula for the input-output equations for models with at least one input was 
given by Meshkat, Sullivant, and Eisenberg~\cite{MeshkatSullivantEisenberg}.
Here we clarify their formula (see Proposition~\ref{prop:i-o} and Remark~\ref{rmk:gcd}), 
and then use that result to explain how input-output equations can be read off from the input-output equations of submodels arising from what we call \textit{output-reachable} subgraphs (Theorem~\ref{thm:ioscc}).  

\begin{table}[th]
\centering
\begin{tabular}{ll}
\hline
{\bf Operation}     	& {\bf Operation preserves identifiability?}       \\  \hline
Add input			&  Yes* (Proposition~\ref{prop:add-in-out})
\\
Add output		&  Yes* (Proposition~\ref{prop:add-in-out})
\\
Add leak			&  Not always (See Example~\ref{ex:add-delete-leak}); Yes, 
				under certain hypotheses (Theorem~\ref{thm:add-leak})
\\
Add edge			&  Not always (See Example~\ref{ex:add-delete-edge})
\\ \hline
Delete input			&  Not always (See Example~\ref{ex:delete-in})
\\
Delete output			&  Not always (See Example~\ref{ex:delete-out})
\\
Delete leak			& Open (See Question~\ref{q:leak}); 
				Yes, under certain hypotheses (Proposition~\ref{prop:removeleaks})
\\
Delete edge			&  Not always  (See Example~\ref{ex:add-delete-edge}); Yes, 
						under certain hypotheses \cite[Theorem 3.1]{singularlocus}
\\
\hline
\end{tabular}
\caption{Operations on linear compartmental models, and whether they preserve identifiability. Here, * pertains to models with at least one input.
\label{table:add-delete} }
\end{table}

\begin{table}[th]
\centering
\begin{tabular}{ll}
\hline
{\bf Operation}     	& {\bf Operation preserves unidentifiability?}       
\\  \hline
Delete input			&  Yes* (Proposition~\ref{prop:add-in-out})
\\
Delete output		&  Yes* (Proposition~\ref{prop:add-in-out})
\\
Delete leak			&  Not always (See Example~\ref{ex:add-delete-leak}); Yes, 
				under certain hypotheses (Theorem~\ref{thm:add-leak})
\\
Delete edge			&  Not always (See Example~\ref{ex:add-delete-edge})
\\ \hline
Add input			&  Not always (See Example~\ref{ex:delete-in})
\\
Add output			&  Not always (See Example~\ref{ex:delete-out})
\\
Add leak			& Open (See Question~\ref{q:leak}); 
				Yes, under certain hypotheses (Proposition~\ref{prop:removeleaks})
\\
Add edge			&  Not always  (See Example~\ref{ex:add-delete-edge}) 
\\
\hline
\end{tabular}
\caption{Operations on linear compartmental models, and whether they
  preserve unidentifiability.  
A model is {\em unidentifiable} if it is not identifiable.
Also, * pertains to models with at least one input.
\label{table:add-delete-uniden} }
\end{table}

We apply our results on input-output equations to investigate whether identifiability is preserved when a leak is added (Theorem~\ref{thm:add-leak}) or removed (Proposition~\ref{prop:removeleaks}).  
The effect on identifiability of adding or deleting a component is not always predictable, as seen in Tables~\ref{table:add-delete} and~\ref{table:add-delete-uniden}.  The rows of these two tables are similar, because, for instance, losing identifiabilty when a leak is added can also be viewed as gaining identifiability when the leak is removed.

\begin{rmk} Tables~\ref{table:add-delete} and~\ref{table:add-delete-uniden} are nearly identical, except that, in the last line, \cite[Theorem 3.1]{singularlocus} does not apply to unidentifiable models: that theorem makes use of the ``singular-locus equation'' which is not defined for unidentifiable models. 
\end{rmk}

 One of the significant contributions of our results is
  that they directly ``transfer'' information about the idenfiability of a model to a related model 
  -- without additional computation.  We illustrate this in
  Proposition~\ref{prop:3-models}, where we quickly conclude that
  three infinite families of models are all identifiable.  Another
  rationale behind our work, hinted at earlier, is to develop tools
  that will enable a future catalogue of identifiable models.

The outline of our work is as follows.
In Section~\ref{sec:background}, we introduce linear compartmental models and define identifiability. 
In Section~\ref{sec:i-o}, we prove our results on input-output equations.  
In Section~\ref{sec:leak}, we prove results on how identifiability is affected by adding or removing inputs, outputs, or leaks; and then we give some related examples in Section~\ref{sec:ex}. 
We conclude with a discussion in Section~\ref{sec:discussion}.

\section{Background} \label{sec:background}
In this section, we recall linear compartmental models, their input-output equations, and the concept of identifiability.  We follow closely the notation in~\cite{singularlocus}.

A {\em linear compartmental model} is defined by a directed graph 
$G = (V,E)$ 
and three sets $In, Out,$ $Leak \subseteq V$.
Each vertex $i \in V$ is a
compartment in the model, while each edge $j \rightarrow i$ represents  
the flow of material from the $j$-th compartment to the
$i$-th compartment. The sets 
$In$, $Out$, and $Leak$ are the
sets of input, output, and leak compartments, 
respectively.  
Each compartment $i \in In$ 
has an external input $u_i(t)$ driving the system, whereas 
each compartment $j \in Out$ 
is measurable. 
A compartment $k \in Leak$ has some constant rate of flow that leaves the system (i.e., does not flow into any other compartment).
We assume that $Out$ is nonempty (as, otherwise, no variables can be observed and hence the model parameters cannot be recovered).

Consistent with the literature, we indicate output compartments by this symbol: \begin{tikzpicture}[scale=0.7]
 	\draw (4.66,-.49) circle (0.05);	
	 \draw[-] (5, -.15) -- (4.7, -.45);	
\end{tikzpicture} .  
Input compartments are labeled by ``in'', and leaks are indicated by outgoing edges.  
For instance, the linear compartmental model in Figure~\ref{fig:model} 
has $In=Out=\{1,3\}$ and $Leak =\{1,4\}$.

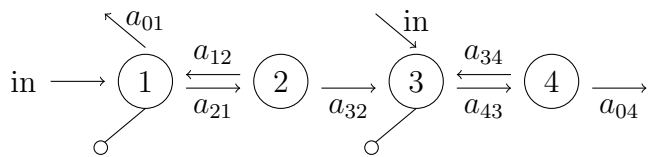
\begin{figure}[ht]
\begin{center}
	\begin{tikzpicture}[scale=1.8]
 	\draw (0,0) circle (0.2);	
 	\draw (1,0) circle (0.2);	
 	\draw (2,0) circle (0.2);	
 	\draw (3,0) circle (0.2);	
    	\node[] at (0, 0) {1};
    	\node[] at (1, 0) {$2$};
    	\node[] at (2, 0) {$3$};
    	\node[] at (3, 0) {$4$};
	 \draw[<-] (0.3, 0.05) -- (0.7, 0.05);	
	 \draw[->] (0.3, -0.05) -- (0.7, -0.05);	
	 \draw[->] (1.3, -0.05) -- (1.7, -0.05);	
	 \draw[<-] (2.3, 0.05) -- (2.7, 0.05);	
	 \draw[->] (2.3, -0.05) -- (2.7, -0.05);	
	 \draw[->] (3.3, -0.05) -- (3.7, -0.05);	
   	 \node[] at (0.5, 0.2) {$a_{12}$};
   	 \node[] at (2.5, 0.2) {$a_{34}$};
   	 \node[] at (0.5, -0.2) {$a_{21}$};
   	 \node[] at (1.5, -0.2) {$a_{32}$};
   	 \node[] at (2.5, -0.2) {$a_{43}$};
   	 \node[] at (3.5, -0.2) {$a_{04}$};
 	\draw (-0.33,-.49) circle (0.05);	
	 \draw[-] (0, -.2 ) -- (-0.3, -.45);	
 	\draw (1.67,-.49) circle (0.05);	
	 \draw[-] (2, -.2 ) -- (1.7, -.45);	
	 \draw[->] (-0.7, 0) -- (-0.3, 0);	
   	 \node[] at (-.9, 0) {in};
	 \draw[<-] (2, .25) -- (1.7, .5);	
   	 \node[] at (2, 0.45) {in};
	 \draw[->] (0, .25) -- (-0.3, .5);	
   	 \node[] at (0, 0.45) {$a_{01}$};
	\end{tikzpicture}
\end{center}
\caption{A linear compartmental model.}
\label{fig:model}
\end{figure}

To every edge $j \rightarrow i$ in $G$, we associate
a 
parameter $a_{ij}$, the 
rate constant for the flow
from compartment-$j$ to compartment-$i$.  
Similarly,
to every leak node $i \in Leak$, we associate a 
parameter $a_{0i}$. 
Letting $n=|V|$,
the {\em compartmental matrix} of a linear compartmental model $(G,In, Out, Leak)$ is the $n \times n$ matrix 
 $A$ 
 with entries as follows:
\[
  A_{ij} 
  ~:=~ \left\{ 
  \begin{array}{l l l}
    -a_{0i}-\sum_{k: i \rightarrow k \in E}{a_{ki}} & \quad \text{if $i=j$ and } i \in Leak\\
        -\sum_{k: i \rightarrow k \in E}{a_{ki}} & \quad \text{if $i=j$ and } i \notin Leak\\
    a_{ij} & \quad \text{if $j\rightarrow{i}$ is an edge of $G$}\\
    0 & \quad \text{otherwise.}\\
  \end{array} \right.
\]

A linear compartmental model $(G, In, Out, Leak)$ defines a system of linear ODEs (with inputs $u_i(t)$) and 
outputs $y_{i}(t)$ given by:
\begin{align} \label{eq:main}
{x}'(t) ~&=~ Ax(t)+u(t) \\ 
y_i(t) ~&=~ x_i(t) \quad \quad \mbox{ for } i \in Out~, \notag
\end{align}
 where $u_{i}(t) \equiv 0$ for $i \notin In$.
 In general, linear compartmental models additionally
   involve parameter scalings of input and output variables.  
Such additional parameters can be accommodated by the differential algebra approaches we use here, but we postpone analysis of such models for future work. 

\begin{eg} \label{ex:ode}
For the model in Figure~\ref{fig:model}, the ODEs~\eqref{eq:main} are given by:
\begin{align} \label{eq:A-ex}
\begin{pmatrix}
x_1' \\
x_2' \\
x_3' \\
x_4' 
\end{pmatrix} 
&~=~
\begin{pmatrix}
- a_{01} - a_{21} & a_{12} & 0 & 0 \\
a_{21} & -a_{12}-a_{32} & 0 & 0\\
0 & a_{32} & -a_{43} & a_{34} \\
0 & 0 & a_{43} & -a_{04} -a_{34}
\end{pmatrix}
\begin{pmatrix}
x_1 \\
x_2 \\
x_3 \\
x_4
\end{pmatrix} +
\begin{pmatrix}
u_1 \\
0 \\
u_3 \\
0
\end{pmatrix}~,
\end{align}
with 
output equations $y_1=x_1$ and $y_3 = x_3$.
\end{eg}

\begin{defn} \label{def:strongly-connected} 
A directed graph $G$ is \textit{strongly connected} if there exists a directed path from each vertex to every other vertex.  A {\em strong component} of a directed graph $G$ is
a strongly connected, induced subgraph of $G$ that is 
maximal with respect to inclusion.
A linear compartmental model $(G, In, Out, Leak)$ is \textit{strongly connected}
if $G$ is strongly connected. 
\end{defn}

\subsection{Input-output equations} \label{subsec:i-o}
{There are many techniques for assessing identifiability of linear
  compartmental models, here we use differential algebra.  One key
  advantage of the differential algebra approach is that it readily
  distinguishes between local and global
  identifiability~\cite{global-id}.  Below we explain how the
  differential algebra approach yields ``input-output equations.''  In
  Section~\ref{sec:identifiability-defn}, we will use such equations to characterize
  identifiability of linear compartmental models.

The general setup is a model $\mathcal{M}$ of the following form:
\begin{align} \label{eq:general-model}
  {x}'(t) &= f(x(t),u(t),p) \\
  \notag
  y(t) &=g(x(t),p)~,
\end{align}
\noindent
where $x(t)$ is the state-variable vector, $u(t)$ is the input vector,
$y(t)$ is the output vector, $p$ is the vector of unknown
(constant) parameters, 
and $f$ and $g$ are polynomials.  
We can therefore view the model equations as differential polynomials
in a differential polynomial ring $R(p)[u,y,x]$, i.e., the ring of
polynomials in $x$, $y$, $u$, and their derivatives, with coefficients
in $R(p)$.  
As the unmeasured state variables $x_i$ cannot be determined, we
use differential elimination to 
eliminate all unknown state variables and their derivatives.   
The resulting equations are only in terms of input variables, output
variables, their derivatives, and parameters, so these equations have the following form:
\begin{align} \label{eq:general-i-o}
\sum_i{c_i(p)\Phi_i(u,y)} =0~.
\end{align}
An equation of the form~\eqref{eq:general-i-o} is called an \textit{input-output equation} for
$\mathcal M$. 

One standard ``reduced'' generating set for these input-output equations is formed by those equations in a \textit{characteristic set} (defined precisely in~\cite{glad})
that do not involve the $x_i$'s or their derivatives.
In a characteristic set, which can be computed using the software {\tt DAISY}~\cite{daisy},
each $\Phi_i(u,y)$ in each input-output equation~\eqref{eq:general-i-o} 
is a differential monomial, i.e., a monomial purely
in terms of input variables, output variables, and their derivatives. The terms $c_i(p)$ are called the coefficients of the input-output equations.  These coefficients can be fixed uniquely by normalizing the input-output equations to make them monic \cite{daisy}.



Returning to linear compartmental models, 
an input-output equation is an equation that holds along every solution of the ODEs~\eqref{eq:main} and 
    involves only the parameters $a_{ij}$, input variables $u_i$,
    output variables $y_i$, and their derivatives.  
A general form of some of these input-output equations
is given in the following result, which is largely due to Meshkat, Sullivant, and Eisenberg~\cite[Theorem 2]{MeshkatSullivantEisenberg} (our 
contribution 
 is explained in Remark~\ref{rmk:gcd}):
\begin{proposition} 
 \label{prop:i-o}
Let $\mathcal M= (G, In, Out, Leak)$ be a  
 linear compartmental model with $n$ compartments and  
 at least one input.
Define $\partial I$ to be the $n \times n$ matrix in which every diagonal entry is the differential operator $d/dt$ and every off-diagonal entry is 0.  
Let $A$ be the compartmental matrix, and let $\left( \partial I-A \right)_{ji}$ denote the $(n-1) \times (n-1)$ matrix obtained from $\left( \partial I-A \right)$ by removing row $j$ and column $i$.
Then, the following equations 
are input-output equations for $\mathcal M$:
 \begin{align} \label{eq:i-o-for-M-general}
 	\det (\partial I -A) y_i ~=~  \sum_{j \in In}  (-1)^{i+j}  \det \left( \partial I-A \right)_{ji} u_j 
		\quad \quad {\rm for~} i \in Out~.
\end{align}
\end{proposition}

We will refer to the input-output equations~\eqref{eq:i-o-for-M-general} as \underline{the} input-output equations. 
(See, for instance, Example~\ref{ex:caution} below.) 


\begin{rmk} \label{rmk:det}
By convention, the determinant of the empty matrix equals 1, so when $n=1$,  we have $\det \left( \partial I-A \right)_{11}:=1$ in the input-output equation~\eqref{eq:i-o-for-M-general}.
\end{rmk}

\begin{rmk} \label{rmk:i-o}
Proposition~\ref{prop:i-o} 
requires  
that $\mathcal{M}$ have at least one input.  This hypothesis, not stated in~\cite[Theorem 2]{MeshkatSullivantEisenberg}, is required in the proof. 
Namely, Cramer's rule is applied to the system $(\partial I - A)x=u$, which requires that the input vector $u$ be a nonzero vector (i.e., at least one input).  
\end{rmk}

\begin{rmk}[Input-output equation for models with no inputs and one output] \label{rmk:1-out}
In spite of Remark~\ref{rmk:i-o},
the input-output equation~\eqref{eq:i-o-for-M-general} generalizes to 
models with no inputs and a single output $y_i$.  
Namely, the input-output equation in this case is just ${\det (\partial I -A)} y_i=0$, due to the fact that the ODE system is $(\partial I - A)x=0$ and thus has a nonzero solution $x$ if and only if ${\det (\partial I -A)} =0$.  This result is seen easily from the \textit{transfer function approach} to finding input-output equations, which is detailed in \cite{distefano-book}.
\end{rmk}

\begin{rmk}[Our contribution to Proposition~\ref{prop:i-o}] \label{rmk:gcd}
In~\cite[Theorem 2]{MeshkatSullivantEisenberg}, 
it is claimed that an input-output equation can be obtained by 
dividing both sides of the input-output equation~\eqref{eq:i-o-for-M-general} by 
  the greatest common divisor (GCD) among the differential polynomials $\det (\partial I -A)$ and the $\det \left( \partial I-A \right)_{ji}$'s (for those $j \in In$).
 The resulting equation, however, is {\em not} always an input-output equation (see Example~\ref{ex:caution}).  
  Our contribution, therefore, is to correct the input-output equation
  stated in~\cite[Theorem 2]{MeshkatSullivantEisenberg}.
 Indeed, a proof of
Proposition~\ref{prop:i-o} is obtained by deleting the last two
sentences (which pertain to dividing by the GCD) of~\cite[proof of Theorem
  2]{MeshkatSullivantEisenberg}.
\end{rmk}

 Remark~\ref{rmk:gcd} motivates the following definition and subsequent question.

\begin{defn} \label{def:gcd}
Let $\mathcal M=(G,In,Out,Leak)$ be a linear compartmental model with compartmental matrix $A$.
Let $i \in Out$.
The {\em input-output GCD} of $y_i$ with respect to $\mathcal M$ is
the GCD of the following set of nonzero differential polynomials:
\[
\left( \{\det (\partial I -A) \} \cup \{ \det \left( \partial I-A \right)_{ji} \mid j \in In \} \right) \setminus \{0 \}~.
\]
\end{defn}

\begin{question} \label{q:gcd}
Consider the equation obtained by dividing 
both sides of the input-output equation~\eqref{eq:i-o-for-M-general} by the 
input-output GCD of $y_i$.
For which models is this equation 
an input-output equation?
\end{question}

An answer to Question~\ref{q:gcd} would help generate simpler
input-output equations than those in~\eqref{eq:i-o-for-M-general}.  
These simpler equations, in turn, may be used to assess whether a
model is identifiable (see Section~\ref{sec:identifiability-defn} below).  

Models for which the input-output GCD is 1 automatically satisfy Question~\ref{q:gcd}, so 
a partial answer to this question was given by 
Meshkat, Sullivant, and Eisenberg, who showed that the GCD 
is 1 for strongly connected models with at least one leak and at least one input~\cite[Corollary 1]{MeshkatSullivantEisenberg}. 
One of our goals here is to generalize that result by removing the requirement of having leaks (Proposition~\ref{prop:i-o-for-strongly-connected} in the next section).
Moreover, even for models that are not strongly connected,
we make progress toward Question~\ref{q:gcd}
by showing that the ``downstream'' components correspond to factors of the GCD
(Theorem \ref{thm:ioscc}).

\begin{eg} \label{ex:caution}
Consider the following model $\mathcal M$ for which $In=Out=\{2\}$:

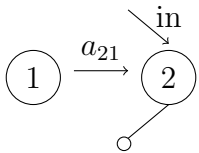
\begin{center}
	\begin{tikzpicture}[scale=1.8]
 	\draw (8,0) circle (0.2);	
 	\draw (9,0) circle (0.2);	
    	\node[] at (8, 0) {1};
    	\node[] at (9, 0) {$2$};
	 \draw[->] (8.3, 0.05) -- (8.7, 0.05);	
   	 \node[] at (8.5, 0.2) {$a_{21}$};
 	\draw (8.67,-.49) circle (0.05);	
	 \draw[-] (9, -.2 ) -- (8.7, -.45);	
	 \draw[<-] (9, .25) -- (8.7, .5);	
   	 \node[] at (9, 0.45) {in};	
\end{tikzpicture}
\end{center}
Following Proposition~\ref{prop:i-o},
the input-output equation of $\mathcal M$ is given by:
\begin{align} \label{eq:y2}
 	{\det (\partial I -A)} y_2 ~=~ 
		 (-1)^{2+2}  \det \left( \partial I-A \right)_{22} u_2 ~,
 \end{align}
where 
\[
	\partial I -A ~=~ 
		\begin{pmatrix}
		d/dt + a_{21} & 0 \\
		-a_{21} & d/dt 
		\end{pmatrix}~.
\]
Therefore, the input-output 
equation~\eqref{eq:y2} is
\begin{align} \label{eq:y2-again}
 d/dt (d/dt+a_{21}) y_2 = (d/dt+a_{21}) u_2~, 
\end{align} 
 that is, 
$y_2'' + a_{21}  y_2' =  u_2' + a_{21} u_2$.

If we divide both sides of the input-output equation
\eqref{eq:y2-again} by the GCD of $ d/dt (d/dt+a_{21})$ and $ (d/dt+a_{21})$, 
which is $ (d/dt+a_{21})$, we obtain $ y_2' = u_2$.  This equation is {\em not} an input-output equation of $\mathcal M$, because the ODEs of $\mathcal M$ satisfy $ y_2' = a_{21}x_1+u_2 \neq u_2$.
\end{eg}

\begin{eg}[Example~\ref{ex:ode}, continued] \label{ex:i-o-running-ex-with-gcd}
Returning to the model in Figure~\ref{fig:model}, with 
$In=Out=\{1,3\}$, we compute the input-output equation involving $y_1$ using Proposition \ref{prop:i-o}. 
We will see that, in this example, dividing by the input-output GCD 
does give an input-output equation even though this need not be the case in general (recall Remark~\ref{rmk:gcd}).
First, for $i=1$, the input-output equation~\eqref{eq:i-o-for-M-general} is
\begin{align} \label{eq:y1gcd}
 	{\det (\partial I -A)} y_1 ~=~ 
		 (-1)^{1+1} { \det \left( \partial I-A \right)_{11}} u_1 
		~+~
		 (-1)^{1+3} { \det \left( \partial I-A \right)_{31}} u_3 ~.
 \end{align}
Examining the compartmental matrix $A$, in equation~\eqref{eq:A-ex},
we see that in the $3 \times 3$ matrix $\left( \partial I-A \right)_{31}$, the upper-right $2 \times 2$ submatrix is the zero-matrix.  So, $ \det \left( \partial I-A \right)_{31}=0$.

Next,  $\left( \partial I-A \right)$ and 
$\left( \partial I-A \right)_{11}$ are block lower-triangular with lower block as follows:
	\begin{align*}
	B~:=~
		\begin{pmatrix}
		d/dt + a_{43} & -a_{34} \\
		-a_{43} & d/dt + a_{04} + a_{34}
		\end{pmatrix}~.
	\end{align*}
Moreover, it is straightforward to check that the GCD of $\det \left( \partial I-A \right)$ and 
$\det \left( \partial I-A \right)_{11}$ is $g_1=\det B$.  
Dividing both sides of the input-output equation~\eqref{eq:y1gcd} yields the following 
equation:
\begin{equation} \label{eq:i-o-1-gcd}
y_1'' + (a_{01}+a_{21}+a_{12}+a_{32}) y_1' + (a_{01}a_{12}+a_{01}a_{32}+a_{21}a_{32})y_1 ~=~ u_1' + (a_{12}+a_{32})u_1~,
\end{equation}
and we can verify, e.g., using the software {\tt DAISY}~\cite{daisy}, that~\eqref{eq:i-o-1-gcd} is an input-output equation for the model.  Later, we will see that we can also obtain this equation via Theorem~\ref{thm:ioscc}, by restricting the model to those compartments that ``flow into'' compartment-1.  

For the other output, $i=3$, the input-output equation~\eqref{eq:i-o-for-M-general} is
\begin{align*}  
\det (\partial I -A) y_3 ~=~ 
		 (-1)^{3+1} \det \left( \partial I-A \right)_{13} u_1 
		~+~
		(-1)^{3+3} \det \left( \partial I-A \right)_{33} u_3 ~,
 \end{align*} 
which simplifies to the following input-output equation:
\begin{align} \label{eq:i-o-3}
	& y_3^{(4)} -(a_{01}+a_{21}+a_{12}+a_{32}+a_{43}+a_{04}+a_{34}) y_3^{(3)}
	\\ \notag
	& \quad 
	+(a_{01}a_{12}+a_{01}a_{32}+a_{21}a_{32}+a_{04}a_{43}+(a_{01}+a_{21}+a_{12}+a_{32})(a_{43}+a_{04}+a_{34})) y_3''
	\\ \notag
	& \quad 
	+((a_{01}a_{12}+a_{01}a_{32}+a_{21}a_{32})(a_{43}+a_{04}+a_{34})+a_{04}a_{43}(a_{01}+a_{21}+a_{12}+a_{32})y_3'
	\\ \notag
	& \quad 
	+ (a_{01}a_{12}+a_{01}a_{32}+a_{21}a_{32}) (a_{04}a_{43}) y_3 
		\\ \notag
	&  
	=~ (a_{21}a_{32}) u_1' + 
		a_{21}a_{32}(a_{04}+a_{34})u_1 \\ \notag
	& \quad	
		+ u_3^{(3)} + (a_{01}+a_{21}+a_{12}+a_{32}+a_{04}+a_{34})u_3''
	\\ \notag
	& \quad 
	+(a_{01}a_{12}+a_{01}a_{32}+a_{21}a_{32}+(a_{04}+a_{34})(a_{01}+a_{21}+a_{12}+a_{32}))u_3'
		\\ \notag
	& \quad 
	+(a_{01}a_{12}+a_{01}a_{32}+a_{21}a_{32})(a_{04}+a_{34})u_3~.
\end{align}

\end{eg}

\subsection{Identifiability} \label{sec:identifiability-defn}
A model is 
\textit{generically structurally identifiable} if from a
generic 
choice of both the inputs
and initial conditions, the parameters of the model can be recovered 
from exact measurements of both the inputs and the
outputs~\cite{Bellman,global-id}. 
Next, we give another standard definition of identifiability, through input-output equations in a
characteristic set (Definition~\ref{defn:identifygeneral}).  

The
intuition behind this definition is as follows.  Consider, for example, the input-output
equation~\eqref{eq:i-o-3}.  If we have many (perfect) measurements of
$y_3, y_3', \dots, y_3^{(4)}$, $u_1, u_1'$, $u_3, u_3',\dots, u_3^{(3)}$, then
solving for the coefficients in the equation, such as
$(a_{01}+a_{21}+a_{12}+a_{32}+a_{43}+a_{04}+a_{34})$, means
solving a linear system.  (So that there is a unique solution, we must avoid 
``structural'' linear dependencies among these measurements; this is
accomplished by using a characteristic set, and a related issue will be discussed in
Remark~\ref{rmk:gleb}.)  Identifiability, therefore, refers to when
the map from the parameters to the coefficients is one-to-one.  

In the following definition, a ``differential monomial term''
is a term of the form $\Phi_i(u,y)$ in
an input-output equations as in~\eqref{eq:general-i-o}.

\begin{defn}\label{defn:identifygeneral}
Let $\mathcal{M}$ be a model, as in \eqref{eq:general-model}, with $P$
parameters. 
Let $\Sigma$ be the set of input-output equations from the characteristic set\footnote{For concreteness, we choose the characteristic set arising from the ranking used by {\tt DAISY}~\cite{daisy}.}
 for $\mathcal M$.  
The  \textit{coefficient map} of $\Sigma$
is the function 
$c:  \R^{P}  \rightarrow \R^{k}$
that is the vector of all non-monic coefficient functions of every 
differential monomial term in every input-output equation in $\Sigma$
(here $k$ is the total number of non-monic coefficients).
Then
$\mathcal{M}$ is:
\begin{enumerate}
	\item \textit{globally identifiable} if $c$ is one-to-one, and is \textit{generically globally identifiable} if $c$ is one-to-one outside a set of measure zero.
	\item {\textit{locally identifiable} if
		around every point in $\R^{P}$ there is an open neighborhood $\mathcal U$ such that 
		 $c : \mathcal U  \rightarrow \R^k$  is one-to-one, and is \textit{generically locally identifiable} if, 
		 outside a set of measure zero, every point in $\R^{P}$ has such an open 
		 neighborhood $\mathcal U$.}
\item \textit{unidentifiable} if $\mathcal M$ is infinite-to-one.
\end{enumerate}
\end{defn}

\begin{eg}[Example~\ref{ex:caution}, continued] \label{ex:caution-continued}
Recall that the model $\mathcal M$ in 
Example~\ref{ex:caution} has input-output equation 
$ y_2'' + a_{21}  y' =  u_2' + a_{21} u_2$, 
comes from a characteristic set, as verified using {\tt DAISY} \cite{daisy}, 
so the coefficient map $c:
\mathbb{R} \to \mathbb{R}^2$ is given by $a_{21} \mapsto (a_{21},a_{21})$.
Hence, $\mathcal M$ is globally identifiable.  
\end{eg}

The following result, which is \cite[Proposition 2]{MeshkatSullivantEisenberg}, is a criterion for identifiability:

\begin{prop} [Meshkat, Sullivant, and Eisenberg] \label{prop:jacobian}
A coefficient map $c: \mathbb{R}^{|E|+|Leak|} \to \mathbb{R}^k$ of a 
linear compartmental model $(G, In,$ $Out, Leak)$ 
is locally one-to-one (that is, outside a set of measure zero, every point in $\mathbb{R}^{|E|+|Leak|} $ has an open neighborhood $\mathcal U$ such that $c:\mathcal U \to \mathbb{R}^k$ is one-to-one)
if and only if 
the Jacobian matrix of $c$,
when evaluated at a generic point, 
has rank equal to $|E| + |Leak|$. 
\end{prop}

In the next section, we prove new results on the form of the
input-output equations and also clarify some of the subtleties that
arise when finding the input-output equations (such as the GCD
mentioned above).  
We also give another definition of identifiability for
  linear compartmental models.
Then in Section \ref{sec:leak}, we prove new results on when adding or removing leaks preserves identifiability.  Finally, in Section \ref{sec:ex}, we demonstrate some examples of the effects of adding or removing inputs, outputs, leaks, or edges.

\section{Results on input-output equations and identifiability} \label{sec:i-o}
In this section, we prove results that relate input-output equations of a model to those of certain submodels (Section~\ref{sec:i-o-submodel}) 
and then investigate how input-output GCDs (from Definition~\ref{def:gcd}) are related to the model's strong components (Section~\ref{sec:i-o-gcd}).

\subsection{Input-output equations, submodels, and identifiability} \label{sec:i-o-submodel}

Our main result, 
Theorem~\ref{thm:ioscc}, implies that, under certain hypotheses, there is an input-output equation involving an output variable 
$y_i$, that corresponds to the input-output equation arising from $y_i$'s \textit{output-reachable} subgraph (see Definition~\ref{defn:outputreachable}).
We must first explain how to restrict a model $\mathcal M$ to such a subgraph.

\begin{defn}  \label{def:restrict} 
For a linear compartmental model $\mathcal{M}=(G, In, Out, Leak)$, 
let $H=(V_H,E_H)$ be an induced subgraph of $G$ that contains at least one output.
The {\em restriction} of $\mathcal M$ to $H$, denoted by $\mathcal{M}_H$,
is obtained from $\mathcal M$ by removing all incoming edges to $H$, retaining all leaks and outgoing edges (which become leaks), and retaining all inputs and outputs in $H$; that is,
	\[
	\mathcal{M}_H ~:=~
	(H,~ In_H,~Out_H,~Leak_H)~,
	\]  
where the input and output sets are 
$In_H:=In \cap V_H$ and 
$Out_H:=Out \cap V_H$, and the leak set is
\[
Leak_H~:=~ \left( Leak \cap V_H \right) \cup 
	\{
	i \in V_H \mid (i,j) \in E(G) ~{\rm for~some}~ j \notin V_H
	\}~.
\]
Additionally, the labels of edges in $H$ are inherited from those of $G$, and labels of leaks are as follows:
\[
	{\rm  label~of~leak~from~}k^{\rm th} {\rm~compartment}~=~ 
	\begin{cases}
		a_{0k} + \sum_{ \{ j \notin V_H \mid (k,j) \in E(G) \}} a_{jk} 
			 & {\rm if~ }k \in Leak \cap V_H\\
		 \sum_{ \{ j \notin V_H \mid (k,j) \in E(G) \}} a_{jk} 
			 & {\rm if~ }k \notin Leak \cap V_H~.
	 	\end{cases}
\]
\end{defn}

\begin{rmk}
A restriction $\mathcal M_H$ is a linear compartmental model, together with leak-labels which may be sums of parameters.  So, $\mathcal M_H$ may have 
more than  $|E_H|+|Leak_H|$ parameters.
\end{rmk}

\begin{eg}[Example~\ref{ex:i-o-running-ex-with-gcd}, continued] \label{eg:scc} 
Returning to the linear compartmental model $\mathcal M$ from Figure~\ref{fig:model}, 
the restriction to the strong component containing compartment-1 and the strong component containing compartment-3 are, respectively, as follows:

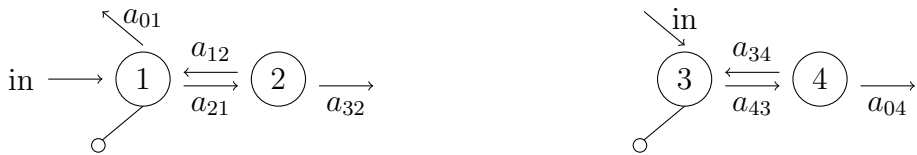
\begin{center}
	\begin{tikzpicture}[scale=1.8]
 	\draw (0,0) circle (0.2);	
 	\draw (1,0) circle (0.2);	
 	\draw (4,0) circle (0.2);	
 	\draw (5,0) circle (0.2);	
    	\node[] at (0, 0) {1};
    	\node[] at (1, 0) {$2$};
    	\node[] at (4, 0) {$3$};
    	\node[] at (5, 0) {$4$};
	 \draw[<-] (0.3, 0.05) -- (0.7, 0.05);	
	 \draw[->] (0.3, -0.05) -- (0.7, -0.05);	
	 \draw[->] (1.3, -0.05) -- (1.7, -0.05);	
	 \draw[<-] (4.3, 0.05) -- (4.7, 0.05);	
	 \draw[->] (4.3, -0.05) -- (4.7, -0.05);	
	 \draw[->] (5.3, -0.05) -- (5.7, -0.05);	
   	 \node[] at (0.5, 0.2) {$a_{12}$};
   	 \node[] at (4.5, 0.2) {$a_{34}$};
   	 \node[] at (0.5, -0.2) {$a_{21}$};
   	 \node[] at (1.5, -0.2) {$a_{32}$};
   	 \node[] at (4.5, -0.2) {$a_{43}$};
   	 \node[] at (5.5, -0.2) {$a_{04}$};
 	\draw (-0.33,-.49) circle (0.05);	
	 \draw[-] (0, -.2 ) -- (-0.3, -.45);	
 	\draw (3.67,-.49) circle (0.05);	
	 \draw[-] (4, -.2 ) -- (3.7, -.45);	
	 \draw[->] (-0.7, 0) -- (-0.3, 0);	
   	 \node[] at (-.9, 0) {in};
	 \draw[<-] (4, .25) -- (3.7, .5);	
   	 \node[] at (4, 0.45) {in};
	 \draw[->] (0, .25) -- (-0.3, .5);	
   	 \node[] at (0, 0.45) {$a_{01}$};
	\end{tikzpicture}
\end{center}

Consider again the input-output equations of $\mathcal M$ for the two outputs, $y_1$ and $y_3$,
which were given in equations~\eqref{eq:i-o-1-gcd} and~\eqref{eq:i-o-3}, respectively. 
Equation~\eqref{eq:i-o-1-gcd} 
is precisely the input-output equation 
 for the model of the restriction above on the left, while 
equation~\eqref{eq:i-o-3}
 involves parameters in the full model and does {\em not} arise from the restriction on the right.  
The reason for this difference, explained below in Theorem~\ref{thm:ioscc}, 
is that the model on the left 
is ``upstream'' of 
the model on the right, but not vice-versa.
\end{eg}

\begin{defn} \label{defn:outputreachable}
For a linear compartmental model $\mathcal{M}=(G, In, Out, Leak)$, let $i \in Out$.
The {\em output-reachable subgraph to} $i$ (or {\em to} $y_i$) 
is
the induced subgraph of $G$ 
containing all vertices $j$ for which there is a directed path in $G$ from $j$ to $i$.
\end{defn}

\begin{eg}[Example~\ref{eg:scc}, continued] \label{eg:outputreachable} 
Returning to the model in Figure~\ref{fig:model}, 
the output-reachable subgraph to $y_1$ is induced by the vertices 1 and 2, so the resulting 
 model of the restriction is the one depicted on the left-hand side in Example \ref{eg:scc}.
On the other hand, the output-reachable subgraph to $y_3$ is induced by all 4 vertices;
 so, the model of the restriction is the original model in Figure~\ref{fig:model}. 
\end{eg}
\begin{rmk} [Output-reachable subgraphs and structural observability]\label{rmk:outputconnectable} A linear compartmental model is \textit{output connectable} \cite{GodfreyChapman} if every compartment has a directed path leading from it to an output 
compartment.  In control theory, a linear compartmental model is \textit{structurally observable} if every state variable $x_i(t)$ can be determined from the inputs $u_j(t)$ and the outputs $y_k(t)$ in some finite (but unspecified) time \cite{distefano-book}. 
A linear compartmental model is structurally observable if and only if it is 
output connectable~\cite{GodfreyChapman}. Thus output-reachable subgraphs 
are structurally observable. The model from Figure~\ref{fig:model} is thus structurally observable.  
\end{rmk}

The following lemma states that an input-output equation of a model of a restriction $\mathcal{M}_H$ is an input-output equation for the full model $\mathcal M$ as long as there are no edges from outside of $H$ into $H$.

\begin{lemma} \label{lem:subgraph} 
For a linear compartmental model 
 $\mathcal M=(G,In, Out,Leak)$,
 let $H=(V_H,E_H)$ be an 
 induced subgraph of $G$ such that there is no directed
edge $i \to j$ 
in G with $i\notin  V_H$ and $j \in V_H$.
 Then every input-output equation of $\mathcal{M}_H$ is an input-output equation of $\mathcal M$.
\end{lemma}

\begin{proof}
There are no directed edges from outside of $H$ into $H$, so the ODEs of $\mathcal M$ are obtained from those of $\mathcal{M}_H$ simply by appending the ODEs for each 
state variable $x_i(t)$ with 
$i \notin V_H$ (this follows from how restrictions are constructed in Definition~\ref{def:restrict}).  Accordingly, any input-output equation of $\mathcal{M}_H$, that is, any equation involving only the input and output variables (and their derivatives) and parameters in $\mathcal{M}_H$ that hold along solutions to the ODEs of $\mathcal{M}_H$, also holds along solutions to the ODEs of $\mathcal{M}$ -- and therefore is also an input-output equation of $\mathcal M$.
\end{proof}

The main result of this section, Theorem~\ref{thm:ioscc},
 states that to obtain an input-output equation involving an output variable $y_i$, it suffices to consider the corresponding restriction $\mathcal{M}_H$ arising from the output-reachable subgraph $H$ to $y_i$.

\begin{theorem} \label{thm:ioscc} 
Let $\mathcal{M}=(G, In, Out, Leak)$
be a linear compartmental model with at least one input.
Let $i \in Out$, 
and assume that there exists a directed path from some input compartment 
to compartment-$i$. 
Let $ H$ denote the output-reachable subgraph to $y_i$, 
and let $A_H$ denote the compartmental matrix for the restriction $\mathcal{M}_H$.  
Then the following is an input-output equation for $\mathcal M$ involving $y_i$:
 \begin{align}  \label{eq:i-o-for-general-model}
 	\det (\partial I -{A}_H) y_i ~=~  (-1)^{i+j} \sum_{j \in In \cap V_H} \det \left( \partial I-{A}_H \right)_{ji} u_j ~,
 \end{align}
where $ \left( \partial I-{A}_H \right)_{ji}$ denotes the matrix obtained from
 $\left( \partial I-{A}_H \right)$ by removing the row corresponding to compartment-$j$ and the column corresponding to compartment-$i$.
Thus, this input-output equation~\eqref{eq:i-o-for-general-model}
involves only the output-reachable subgraph to $y_i$.
\end{theorem}

\begin{proof} 
The set of input compartments of $\mathcal{M}_H$ is $In \cap V_H$ 
 (Definition~\ref{def:restrict}), 
 and now this result follows directly from 
Lemma~\ref{lem:subgraph} and Proposition~\ref{prop:i-o}.
\end{proof} 

Next, we use the the input-output equations given in
Theorem~\ref{thm:ioscc}, in place of those from a characteristic set, 
to give another definition of identifiability (for linear compartmental models). 
Much like in \cite{MeshkatSullivantEisenberg}, we call this notion
``identifiability from the coefficient map.''

\begin{defn}\label{defn:identify}
Let $\mathcal{M}=(G, In, Out, Leak)$ be a linear compartmental model.
Consider the coefficient map
$c:  \R^{|E| + |Leak|}  \rightarrow \R^{k}$
arising from the 
the input-output equations in~\eqref{eq:i-o-for-general-model} 
(here $k$ is the total number of non-monic coefficients). 
Then
$\mathcal{M}$ is:
\begin{enumerate}
	\item \textit{globally identifiable from the coefficient map} 
          if $c$ is one-to-one, and is 
          \textit{generically globally identifiable from the
            coefficient map} 
          if $c$ is one-to-one outside a set of measure zero.
	\item {\textit{locally identifiable from the coefficient map} if
		around every point in $\R^{|E| + |Leak|}$ there is an open neighborhood $\mathcal U$ such that 
		 $c : \mathcal U  \rightarrow \R^k$  is one-to-one,
                 and is \textit{generically locally identifiable from the coefficient map} if, 
		 outside a set of measure zero, every point in $\R^{|E| + |Leak|}$ has such an open 
		 neighborhood $\mathcal U$.}
\item \textit{unidentifiable from the coefficient map} if $\mathcal M$ is infinite-to-one.
\end{enumerate}
\end{defn}
\begin{rmk} \label{rmk:2-defns} 
In all examples we have seen, the two notions of identifiability,
Definitions~\ref{defn:identifygeneral} and~\ref{defn:identify}, are
the same. Indeed, we conjecture that the two definitions are equivalent.  
\end{rmk}

\begin{rmk} \label{rmk:gleb} 
Definition~\ref{defn:identify} is based on the input-output equations
from Theorem~\ref{thm:ioscc}.  Here we show that the choice of
input-output equations matters; the wrong choice can lead to erroneous
conclusions.  
As an example of what can go wrong, consider the following model,
which Gleb Pogudin and Peter Thompson brought to our attention:
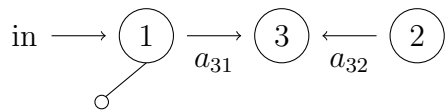
\begin{center} 
	\begin{tikzpicture}[scale=1.8]
	\draw (0,0) circle (0.2);	
 	\draw (1,0) circle (0.2);	
        \draw (2, 0) circle (0.2);
    	\node[] at (0, 0) {1};
    	\node[] at (1, 0) {$3$};
    	\node[] at (2, 0) {$2$};
 	\draw (-0.33,-.49) circle (0.05);	
	 \draw[-] (0, -.2 ) -- (-0.3, -.45);	
	 \draw[->] (-0.7, 0) -- (-0.3, 0);	
   	 \node[] at (-.9, 0) {in};
	 \draw[->] (0.3, 0) -- (0.7, 0);	
	 \draw[<-] (1.3, 0) -- (1.7, 0);	
  	 \node[] at (0.5, -0.2) {$a_{31}$};
  	 \node[] at (1.5, -.2) {$a_{32}$};
	\end{tikzpicture}
\end{center}


The model is unidentifiable, as the output variable $y_1=x_1$ cannot
``see'' $a_{32}$.  Indeed, the model is unidentifiable from the
coefficient map, as seen from the following input-output equation from Theorem~\ref{thm:ioscc}:
\begin{align} \label{eq:i-o-1}
  y_1'+a_{31}y_1=u_1~.
\end{align}

On the other hand, if we instead consider the coefficient map from the
following input-output equation from Proposition~\ref{prop:i-o}:
 \begin{align} \label{eq:i-o-2}
 y_1^{(3)} +(a_{31} +a_{32})  y_1'' + a_{31}a_{32} y_1' =  u_1'' + a_{32} u_1'~,
 \end{align}
this coefficient map is one-to-one (although the model is unidentifiable).  We therefore can not use
input-output equations from Proposition~\ref{prop:i-o} to define
identifiability, and in general we must be cautious regarding which
input-output equations to analyze.  

Recall from Remark~\ref{rmk:gcd} that, in general, dividing an
input-output equation by the input-output GCD might not yield another
input-output equation.  Nevertheless, 
in this example, dividing the input-output equation~\eqref{eq:i-o-2} by the input-output GCD of $y_1$,
which is $d/dt(d/dt+a_{32})$,
yields an input-output equation, the one in~\eqref{eq:i-o-1}.  
So, in
this case, dividing by the input-output GCD yields the correct
input-output equation for assessing identifiability.
(In~\cite{MeshkatSullivantEisenberg}, identifiability is defined in terms of
such equations obtained by dividing by the GCDs.)

Finally, we can use this example to illustrate why the input-output
equations in Proposition~\ref{prop:i-o} are generally not
``reduced'' enough for assessing identifiablity.  The input-output
equation~\eqref{eq:i-o-1} gives a linear dependence among the 
variables
$y_1'$,
$y_1$, and $u_1$, so that (by taking derivatives) there is a linear
dependence among 
$y_1''$,
$y_1'$, and $u_1'$.
This dependence implies that the
coefficients in~\eqref{eq:i-o-2} can {\em not} be recovered from data.  
Indeed, we emphasize that we can always 
perform algebraic operations on ``reduced'' input-output equations to
obtain ``non-reduced''
input-output equations (which should not be used for assessing
identifiability, to avoid erroneous results).  For instance,
by scaling $y_1'+a_{31}y_1=u_1$ by $a_{32}$, adding this
equation to the derivative of $y_1'+a_{31}y_1=u_1$, and then taking
the derivative, we obtain the input-output equation~\eqref{eq:i-o-2}.
\end{rmk}

\begin{rmk}[Output-reachable subgraphs and algebraic observability] A model (linear or nonlinear) is {\em algebraically observable} if every state variable can be recovered from observation of the input and output alone \cite{meshkat-rosen-sullivant}, a more general notion of the control theory concept of ``structural observability.''  
A model is algebraically observable if and only if the sum of the
differential orders of all input-output equations from a 
characteristic set
is equal to the number of state variables, 
$n$~\cite[Proposition 5]{glad}. 
Theorem \ref{thm:ioscc} verifies that, for the case of a single output $y_i$, 
the input-output equation~\eqref{eq:i-o-for-general-model}
for the restriction to the output-reachable subgraph $H$ to $y_i$ has differential order $n=|V(H)|$, and thus the restriction is algebraically observable.
\end{rmk}


\begin{eg}[Example~\ref{eg:scc}, continued] \label{ex:i-o-running-ex}
Returning to the model in Figure~\ref{fig:model}, with 
$In=Out=\{1,3\}$, we compute the input-output equations~\eqref{eq:i-o-for-general-model}. First, for $i=1$, we begin with:
\begin{align} \label{eq:y1}
\det (\partial I -{A_{H_1} }) y_1 ~=~ 
		 (-1)^{1+1} \det \left( \partial I-{A}_{H_1} \right)_{11} u_1 ~,
 \end{align}
where ${A}_{H_1}$ comes from the restriction to the output-reachable subgraph to $y_1$:
\begin{align*}
{A}_{H_1}~:=~
\begin{pmatrix}
	-a_{01} - a_{21} & a_{12} \\
		a_{21} & -a_{12} - a_{32}
		\end{pmatrix}~.
\end{align*}
The resulting input-output equation~\eqref{eq:y1} 
is exactly the one displayed earlier in~\eqref{eq:i-o-1-gcd}.

For the other output, $i=3$, 
the output-reachable subgraph to $y_3$ is the full graph, and so 
the input-output equation~\eqref{eq:i-o-for-general-model} is 
the one given earlier in~\eqref{eq:i-o-3}.
\end{eg}

Theorem~\ref{thm:ioscc} allows us to prove two results on identifiability
 (Corollaries~\ref{cor:unidentifiable} and \ref{cor:obscomp}).

\begin{corollary} \label{cor:unidentifiable}
Let $\mathcal{M}=(G, In, Out, Leak)$ be a linear compartmental model
with at least one input 
such that there is a compartment $j$ such that 
(1) $j$ is {\em not} in any output-reachable subgraph of $G$ and
(2)  $j$ is a leak compartment or there is a directed edge
 $ j \to k$ 
out of $j$. 
Then $\mathcal{M}$ is unidentifiable from the coefficient map.
\end{corollary}

\begin{proof}
For our model, 
Theorem~\ref{thm:ioscc} implies that 
the input-output equations
in~\eqref{eq:i-o-for-general-model},
do {\em not} involve 
the leak parameter $a_{0j}$ (if $j$ is a leak compartment) or the edge
parameter $a_{kj}$ (if $j \to k$ is an edge).
Hence, $\mathcal M$ is unidentifiable from the
  coefficient map.
\end{proof}

\begin{eg} \label{ex:uniden}
In the following model $\mathcal M$, compartment-2 is a leak compartment that
is not in the output-reachable subgraph to the (unique) output in compartment-1:
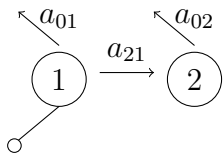
\begin{center}
	\begin{tikzpicture}[scale=1.8]
 	\draw (8,0) circle (0.2);	
 	\draw (9,0) circle (0.2);	
    	\node[] at (8, 0) {1};
    	\node[] at (9, 0) {$2$};
	 \draw[->] (8.3, 0.05) -- (8.7, 0.05);	
   	 \node[] at (8.5, 0.2) {$a_{21}$};
 	\draw (7.67,-.49) circle (0.05);	
	 \draw[-] (8, -.2 ) -- (7.7, -.45);	
	 \draw[->] (8, .25) -- (7.7, .5);	
   	 \node[] at (8, 0.45) {$a_{01}$};	
	 \draw[->] (9, .25) -- (8.7, .5);	
   	 \node[] at (9, 0.45) {$a_{02}$};	
	 \end{tikzpicture}
\end{center}
Hence, by Corollary~\ref{cor:unidentifiable}, $\mathcal M$ is
unidentifiable from the coefficient map.  
Indeed, the 
output-reachable subgraph is induced by compartment-1, so by Theorem~\ref{thm:ioscc},
yields the input-output equation
$ y_1' + (a_{01}+a_{21}) y_1 = 0$, which does not involve the leak parameter $a_{02}$.
\end{eg}


The next result states that the ``observable component'' submodel (Definition~\ref{def:observable-comp}) of an identifiable model is always identifiable.

\begin{defn} \label{def:observable-comp}
For a linear compartmental model $\mathcal{M}$, 
 the \textit{observable component} is the union of all output-reachable subgraphs to all outputs $y_i$ in the model.
\end{defn}

\begin{corollary} \label{cor:obscomp} 
Let $\mathcal{M}$ 
 be a linear compartmental model with at least one input. Let $H$ denote the observable component.
If $\mathcal M$ is generically globally (respectively, locally)
identifiable from the coefficient map, then so is $\mathcal{M}_H$.
\end{corollary}

\begin{proof} 
Assume $\mathcal{M}=(G, In, Out, Leak)$ is identifiable
from the coefficient map.
Let $H$ be the observable component, that is, the union of all
output-reachable subgraphs to outputs $y_i$ (for $i \in Out$).  
By construction, for $i \in Out$, the output-reachable subgraph to
$y_i$ in $G$ is the same as that in $H$.  Hence, the 
input-output equations~\eqref{eq:i-o-for-general-model}
for $\mathcal{M}$ are the same as those for $\mathcal{M}_H$.  
Thus, $\mathcal{M}_H$ is identifiable from the coefficient map.
\end{proof}

\subsection{Input-output GCDs} \label{sec:i-o-gcd}

We now return to Question~\ref{q:gcd}, which concerns input-output GCDs (Definition~\ref{def:gcd}).
Recall that in~\cite[Theorem 2]{MeshkatSullivantEisenberg}, each input-output equation was divided by the corresponding GCD, and also this GCD was proven to
be $1$ in the case of strongly connected models with at least one leak~\cite[Corollary 1]{MeshkatSullivantEisenberg}.  This motivates the following question:

\begin{question} \label{question:gcdone} 
For which models is every input-output GCD equal to 1?
\end{question}

In this subsection, we strengthen~\cite[Corollary 1]{MeshkatSullivantEisenberg}
to allow for strongly connected models without leaks (Proposition~\ref{prop:i-o-for-strongly-connected}). 
Then we elucidate some of the factors of the input-output GCD, and thereby make
progress toward finding the full form of this GCD (Proposition~\ref{prop:gcdscc}).  
As a consequence, we find a necessary condition for the GCD to be $1$ (Corollary~\ref{cor:in-out-reachable}).

\begin{prop}[Input-output GCDs for strongly connected models] \label{prop:i-o-for-strongly-connected} 
For a 
 strongly connected 
 linear compartmental model with 
 at least one input, every input-output GCD is 1.
\end{prop}

\begin{proof}
If $\mathcal M$ has at least 1 leak, this result is \cite[Corollary 1]{MeshkatSullivantEisenberg}.

Assume that $\mathcal M$ has no leaks.  We consider first the case of $n=1$ compartment, with an input and output at that compartment.  There are no leaks or edges, and hence no parameters.  
The input-output equation from Proposition~\ref{prop:i-o} is 
$ y_1' = u_1$, so the input-output GCD is 1.

Now assume that $n \geq 2$ (and $\mathcal M$ has no leaks).  Let $A$ be the compartmental matrix.  
By definition, 
 the input-output GCD of an output variable $y_i$, denoted by 
 $g_i$, is the GCD among the polynomials $\det (\partial I -A)$ and 
the $\det \left( \partial I-A \right)_{ji}$'s for $j \in In$.
Our goal is to prove that $g_i=1$ for all $i \in Out$.

For ease of notation, we consider instead the characteristic polynomial $\det \left( \lambda I-A \right)$ and the related polynomials 
$\det \left( \lambda I-A \right)_{ji}$ for $j \in In$.  We must show that their GCD is 1.

By hypothesis, $\mathcal M$ is strongly connected and has no leaks, so $A$ is the negative of the Laplacian matrix of a strongly connected directed graph $G=(V,E)$.  So, by following the argument in~\cite[Proof of Theorem 3]{MeshkatSullivantEisenberg}, 
a factorization of  $\det (\lambda I -A)$ into two irreducible polynomials is given by $\lambda \cdot \left( \det (\lambda I -A)/ \lambda \right)$.  (Here, the $n \geq 2$ assumption is used.)

Therefore, we need only show that neither (i) $\lambda$ nor (ii) $\det (\lambda I -A)/\lambda $
divides any of the $\det \left( \lambda I-A \right)_{ji}$'s. 
For (i), we must show that $\det \left( \lambda I-A \right)_{ji}|_{\lambda = 0}$ is a nonzero polynomial. Indeed, 
 $\det \left( \lambda I-A \right)_{ji}|_{\lambda = 0} = \det (-A)|_{ji}$, which equals (up to sign) the $(i,j)$-th cofactor of the Laplacian matrix of the strongly connected graph $G$.  
 This determinant, after setting every edge parameter  
 $a_{kl}$ equal to 1, is precisely (by the Matrix-Tree Theorem) the number of directed spanning trees of $G$ rooted at $i$, and this number is at least 1 (because $G$ is strongly connected).  So, $\det \left( \lambda I-A \right)_{ji}|_{\lambda = 0}$ is indeed nonzero.
 
Now we consider (ii).  Viewing  $\det (\lambda I -A)/\lambda $ as a univariate polynomial in $\lambda$, its leading term is $\lambda^{n-1}$.  As for  $\det \left( \lambda I-A \right)_{ji}$, we first assume that $i \neq j$. Then the matrix $\left( \lambda I-A \right)_{ji}$ contains only $n-2$ $\lambda$'s in its entries (both the $i$-th and $j$-th diagonal entries of $\left( \lambda I-A \right)$ were removed). So, in this case, $\det \left( \lambda I-A \right)_{ji}$ has degree less than $n-1$ and therefore can not be divisible by the degree-$(n-1)$ polynomial $\det (\lambda I -A)/\lambda $. 

In the remaining case, when $i=j$, it is straightforward to check that the coefficient of $\lambda^{n-2}$ in 
$\det (\lambda I -A)/\lambda $ equals the following sum over all edges in the graph $G$:
\begin{align} \label{eq:sum-edges}
	\sum_{(k,l) \in E} a_{lk}~.
\end{align}
Similarly, the coefficient of $\lambda^{n-2}$ in $\det \left( \lambda I-A \right)_{ji}$ (when $i=j$) equals 
the sub-sum, over edges that do {\em not} originate at $i$, namely,
$\sum_{(k,l) \in E, ~k \neq i} a_{lk}$.   
This sum is a strict sub-sum of the sum~\eqref{eq:sum-edges}, as $G$ is strongly connected.
So, as desired,
 $\det (\lambda I -A)/\lambda $ does not divide
 $\det \left( \lambda I-A \right)_{ji}$.
\end{proof}

\begin{eg}[Example~\ref{eg:scc}, continued] \label{ex:illustrate-gcd=1}
The model displayed on the left-hand side of
Example~\ref{eg:scc} is strongly connected, with input and output in
compartment-1 only.  
It is straightforward to check that $\det (\partial I-A) = 
(d/dt +a_{12}+a_{32}) (d/dt+a_{01}+a_{21})-a_{12}a_{21}$, and $\det(\partial
I-A)_{11}= (d/dt +a_{12}+a_{32})$.  
The GCD of these two differential polynomials is 1, which is
consistent with 
Proposition~\ref{prop:i-o-for-strongly-connected}.
\end{eg}

We next consider the case when $\mathcal M$ need not be strongly connected.  
The next result shows that the input-output GCD of $y_i$ is 
a multiple of certain ``upstream'', ``neutral'', and ``downstream'' components that do {\em not} contain an input leading to the 
compartment-$i$.

\begin{defn} \label{def:in-out-reach}
Let $\mathcal{M}=(G, In, Out, Leak)$  be a linear compartmental model. 
Let $i \in Out$ be such that there exists a directed path in $G$ from some $j \in In$ to $i$.
The set of compartments lying along such paths induces the {\em input-output-reachable subgraph to} $i$ (or {\em to} $y_i$); more precisely, this is the induced subgraph of $G$ with vertex set containing $i$, every $j \in In$ such that there is a directed path from $j$ to $i$, and every compartment passed through by at least one such path.
\end{defn}

\begin{rmk}[Input-output-reachable subgraphs and structural controllability/observability] \label{rmk:inputconnectable} A linear compartmental model is \textit{input connectable} \cite{GodfreyChapman} if every compartment has a path leading to it originating from an input compartment.  In control theory, a linear compartmental model is \textit{structurally (completely) controllable} if an input can be found that transfers the state variables $x_i(t)$ from any initial state to any specified final state in finite time \cite{distefano-book}.  A \textit{trap} is a strongly connected set of compartments from which no paths exists to any compartment outside the trap, including the environment (i.e. leaks) \cite{GodfreyChapman}.  
A linear compartmental model is structurally controllable if and only if it is input connectable and it is possible to find disjoint paths starting at input compartments and such that every trap has a compartment at the end of one such path~\cite{GodfreyChapman}.  This means our notion of an \textit{input-output-reachable subgraph to $y_i$} corresponds to a structurally observable and structurally controllable model (in our case, there is either a single trap corresponding to the strongly connected component containing $y_i$ or no traps if that strongly connected component contains a leak).  Note that this conclusion is stronger than our conclusion in Remark \ref{rmk:outputconnectable}, where the output-reachable subgraph to $y_i$ is structurally observable.  Being both structurally observable and structurally controllable means, from \cite{GodfreyChapman}, that no model with a smaller number of compartments can be found that will fit the input-output data induced by this subgraph. In addition, no model with a larger number of compartments can be both structurally controllable and structurally observable, by definition.  Thus, our notion of an input-output-reachable subgraph to $y_i$ has the desirable quality from control theory of being the maximal subgraph corresponding to a model that is both structurally controllable and structurally observable.  
\end{rmk}

The following example motivates our next result.
\begin{eg}[Example~\ref{ex:i-o-running-ex}, continued] \label{ex:gcd-factor}
We return to the model in Figure~\ref{fig:model}, where 
$In=Out=\{1,3\}$. 
The input-output reachable subgraph to $i=1$, which
we denote by $\overline H$, is induced by compartments $1$ and $2$.
Let $\overline{H}^c$ denote the subgraph induced by the remaining
compartments, $3$ and $4$.  Letting $A$ denote the compartmental
matrix for $\mathcal M$, we have:
\begin{align} \label{eq:full-matrix}
  \det (\partial I - A) ~=~ \det
  \left( 
    \begin{array}{cccc}
      d/dt+a_{01}+a_{21} & -a_{12}& 0 & 0 \\
      -a_{21} & d/dt+a_{12}+a_{32} & 0 & 0 \\
      0 & -a_{32} & d/dt+a_{43} & -a_{34} \\
      0 & 0 & -a_{43} & d/dt+a_{34}+a_{04}\\
    \end{array}
    \right)~.
\end{align}
Using~\eqref{eq:full-matrix}, it is straightforward to check that 
both 
$ \det (\partial I - A)$ and 
$\det (\partial I - A_{\overline{H}^c} )_{11}$ 
-- where $A_{\overline{H}^c} $ is the compartmental matrix of the
restriction $\mathcal{M}_{\overline{H}^c} $ --
are multiples of the following:
\[
 \det (\partial I - A_{\overline{H}^c}) = \det \left(
   \begin{array}{cc}
      d/dt+a_{43} & -a_{34} \\
      -a_{43} & d/dt+a_{34}+a_{04}\\
   \end{array}
 \right) ~.
\]
Also, $\det (\partial I - A_{\overline{H}^c} )_{31}=0$. 
Thus, for this model, the input-output GCD of $y_1$ is a multiple of 
$\det (\partial I - A_{\overline{H}^c})$.  The following result shows
that this observation generalizes.
\end{eg}

\begin{prop} \label{prop:gcdscc} 
Let $\mathcal{M}=(G, In, Out, Leak)$  be a linear compartmental model. 
Let $i \in Out$ be such that there exists a directed path in $G$ from some $j \in In$ to $i$.
Let $\overline H$ denote the  input-output-reachable subgraph
of $G$ to $i$.
Let $\overline{H}^c$ denote the subgraph induced by all compartments
of $G$ that are {\em not} in $\overline H$, and let $A_{\overline{H}^c}$ denote the compartmental matrix of the
restriction $\mathcal{M}_{\overline{H}^c}$.
Then the input-output GCD of $y_i$ is a multiple of $\det (\partial I - A_{\overline{H}^c})$.
\end{prop}

\begin{proof} 
If $G$ is strongly connected (that is, $G=\overline{H}$), then the result holds trivially.
  
Assume $G$ is not strongly connected.  
Consider all strong components $C$ of $G$ that are not in $\overline H$ and also are 
``upstream'' of $\overline H$, that is, there exists a directed path in $G$ from $C$ to $\overline H$.  Let 
$U$ be the subgraph of $G$ induced by all such strong components. 

Next, let $D$ denote the subgraph of $G$ induced by all strong components of $G$ that are not in $\overline H$ nor in $U$.  (So, $D$ includes ``downstream'' components of $\overline H$.)  

By construction, 
the input-output reachable subgraph $\overline H$ is a union of strong components of $G$.
So, the vertices of $U$, $D$, and $\overline H$ partition the vertices of $G$.

We claim that there are no directed edges
(i)~from a compartment in $\overline H$ to one in $U$, 
(ii)~from a compartment in $D$ to one in $\overline H$, nor
(iii)~from a compartment in $D$ to one in $U$.  
For (i) and (ii), this claim follows from the definition of strong component and by construction of $\overline H$, $U$, and $D$.  For (iii), such an edge would violate the definition of $U$: there would be a strong component in $D$ that should have been in $U$.

Reorder the compartments so that those in $U$ come first, and then those in $\overline H$, and finally those in $D$.  
The lack of edges from $\overline H$ to $U$, from $D$ to $\overline H$, and from $D$ to $U$, yields the following block lower-triangular form
for the compartmental matrix for $\mathcal M$:
\begin{align} \label{eq:A-matrix}
A~=~
\left( 
\begin{array}{c@{}c@{}c}
 \left[\begin{array}{c}
         A_U \\      \end{array}\right] & \mathbf{0} & \mathbf{0} \\     
         \star & \left[\begin{array}{c}
                       A_{\overline H}  \\                          \end{array}\right] & \mathbf{0} \\   
                        \star & \star & \left[ \begin{array}{c}
                                   A_D\\                                      \end{array}\right] \    \end{array}\right)~.
\end{align}
By construction, the diagonal blocks $A_U$, $A_{\overline H}$, and $A_D$ 
are the compartmental matrices of the restriction of $\mathcal M$ to, respectively,
$U$, $\overline H$, and $D$.  
By construction, the compartmental matrix for $\overline{H}^c$ has the following form:
\begin{align*} 
A_{\overline{H}^c}~=~
\left( 
\begin{array}{c@{}c}
 \left[\begin{array}{c}
         A_U \\      \end{array}\right] & \mathbf{0} \\     
                        \star  & \left[ \begin{array}{c}
                                   A_D\\                                      \end{array}\right] \    \end{array}\right)~.
\end{align*}
So, 
\begin{align} \label{eq:decompose-dA-c-I}
	\det ( \partial I - A_{\overline{H}^c} ) ~=~ \det ( \partial I - A_{U} )~
						  \det ( \partial I - A_{D} )~.
\end{align}
Similarly, the following equality follows from~\eqref{eq:A-matrix}:
\begin{align} \label{eq:decompose-dA-I}
	\det ( \partial I - A ) ~=~ \det ( \partial I - A_{U} )~
						 \det ( \partial I - A_{\overline H} )~
						  \det ( \partial I - A_{D} )~.
\end{align}

Recall that the input-output GCD of $y_i$ is the GCD of the determinants 
$\det (\partial I -A)$ and the $\det (\partial I -A)_{ji}$'s (for $j \in In$).  So, 
to show that this GCD is a multiple of $\det (\partial I - A_{\overline{H}^c})$, 
using equations~\eqref{eq:decompose-dA-c-I} and \eqref{eq:decompose-dA-I}, 
we need only show the following claims for every $j \in In$:
\begin{enumerate}[label=(\roman*)]
	\item if $j$ is {\em not} a compartment in $\overline H$, then 
		$\det (\partial I -A)_{ji}=0$.
	\item  if $j$ is 
		in $\overline H$, then 
		$\det (\partial I -A)_{ji}= \det ( \partial I - A_{U} ) 
						 \det ( \partial I - A_{\overline H} )_{ji}
						  \det ( \partial I - A_{D} )$.	
\end{enumerate}
To prove claim (i), assume that $j \in In$ is not in $\overline H$.  
We know that $j$ is in $D$, because compartments in $U$ are not input compartments (otherwise they would be in $\overline H$).  
Removing column-$i$ from the block lower-triangular 
 matrix $\partial I - A$ (one of the columns involving the $(\partial I - A_{\overline H})$-block) drops the rank by 1, and then removing row-$j$ (one of the rows involving the 
 $(\partial I - A_{D})$-block) drops the rank again by 1.  So, $\det (\partial I -A)_{ji}=0$.

For claim (ii), assume that $j \in In$ is in $\overline H$.  In this case, 
both row-$j$ and column-$i$ of $\partial I - A$ involve the block $A_{\overline H}$, so the 
submatrix $ (\partial I -A)_{ji}$ is also block lower-triangular, with upper-left block equal to $ ( \partial I - A_{U} ) $ and lower-right block $( \partial I - A_{D} ) $.   Therefore, 
 $\det (\partial I -A)_{ji}$ factors as claimed in (ii).
\end{proof}   

The following result follows directly from Proposition~\ref{prop:gcdscc}:

\begin{corollary}  \label{cor:in-out-reachable}
Consider a linear compartmental model $\mathcal M=(G,In,Out,Leak)$ with at least one input.  If the input-output GCD of an output variable $y_i$ is 1, then 
 the input-output-reachable component to $y_i$ is the full graph $G$.
\end{corollary}

\section{Identifiability results on adding or removing inputs, outputs, or leaks} \label{sec:leak}
In this section, we show that identifiability is preserved when 
inputs or outputs are added to a model (Proposition~\ref{prop:add-in-out}), 
or, under certain hypotheses, when a leak is added (Theorem~\ref{thm:add-leak}) or removed (Proposition~\ref{prop:removeleaks}).  For a result on which edges can be safely removed without losing identifiability, we refer the reader to~\cite[Theorem~3.1]{singularlocus}.

Our first result is a 
proof of the widely accepted fact that, if a model is identifiable, adding inputs or outputs preserves identifiability.  We include its proof for completeness, as we could not find a formal proof in the literature.


\begin{proposition}[Adding inputs or outputs] \label{prop:add-in-out}
Let $\mathcal M=(G,In,Out,Leak)$ be a linear compartmental model that 
has at least one input, 
and let $\widetilde{\mathcal{M}}$ be a model obtained from $\mathcal M$ by adding an input or an output (that is, $In$ or $Out$ is enlarged by one compartment). 
If $\mathcal M$ is generically globally (respectively, locally)
identifiable from the coefficient map,
then so is $\widetilde{\mathcal{M}}$. 
\end{proposition}

\begin{proof}
Adding an output yields a new input-output equation, while adding an input 
adds coefficients to existing input-output equations, by
  Theorem~\ref{thm:ioscc}. 
Either case extends the coefficient map
$c:  \R^{|E| + |Leak|}  \rightarrow \R^{k}$
to some 
$\widetilde c =(c, \overline{c}):  \R^{|E| + |Leak|}  \rightarrow \R^{k+ \ell}$~. 
Specifically, in the case of adding an output, $\overline{c}$ corresponds to the coefficients 
of the additional equation from the new output, while in the case of adding an input, 
$\overline{c}$ corresponds to coefficients of the new input variable $u_j$ and its derivatives on the right-hand side of the 
input-output equations~\eqref{eq:i-o-for-general-model}.
If $c$ is generically one-to-one (respectively, generically finite-to-one)
then so is $\widetilde c$.  
\end{proof}

\begin{rmk} The converse to Proposition \ref{prop:add-in-out} is not true in general (see Examples \ref{ex:delete-in} and \ref{ex:delete-out} in the next section).  
In other words, identifiability can be lost by removing inputs or outputs.  
Accordingly, there are minimal sets of outputs for identifiability; 
these sets were investigated by 
Anguelova, Karlsson, and Jirstrand \cite{Anguelova}.  
However, finding minimal sets of inputs for a given output remains an open question.
\end{rmk}

The next results investigate how adding or removing a leak affects identifiability.

\begin{theorem}[Adding one leak]
\label{thm:add-leak}
Let $\mathcal M$ be a linear compartmental model that is strongly connected and 
has at least one input and no leaks.
Let $\widetilde{\mathcal{M}}$ be a model obtained from $\mathcal M$ by adding one leak.  If $\mathcal M$ is generically locally identifiable from the coefficient map,
then so is $\widetilde{\mathcal{M}}$. 
\end{theorem}

\begin{proof}
Let $\mathcal M=(H,In,Out,Leak)$, with  $H=(V,E)$, be a strongly connected linear compartmental model with $n$ compartments, at least 1 input, and no leaks.  
Let $\widetilde{\mathcal{M}}$ be obtained from $\mathcal M$ by adding a leak from one compartment, which we may assume is compartment-$1$.  
 
As in~\eqref{eq:main}, we write the ODEs of $\mathcal M$ and $\widetilde{\mathcal{M}}$, respectively, as follows:
\[
		\frac{dx(t)}{dt}=A~x(t)+u(t) \quad \quad {\rm and}
		\quad \quad
		 \frac{dx(t)}{dt}= \widetilde A~x(t)+u(t) ~.
\]
Here $A$ is the $n \times n$ compartmental matrix for $\mathcal M$, 
so the 
column sums are 0 (because $A$ is the negative of the Laplacian matrix of a graph).  
Also, $\widetilde A$ is obtained from $A$ by adding $-a_{01}$ (where $a_{01}$ is the new leak parameter) to the (1,1)-entry.

The following are the input-output equations for
  $\mathcal M$
arising from Theorem~\ref{thm:ioscc}: 
 \begin{align} \label{eq:i-o-for-M}
 	\det (\partial I -A) y_i ~=~  (-1)^{i+j} \sum_{j \in In} \det \left( \partial I-A \right)_{ji} u_j 
	\quad \quad {\rm for~} i \in Out~.
 \end{align}
We know that $| Out| \geq 1$, because $\mathcal M$ is identifiable.  

As for $\widetilde{\mathcal{M}}$, 
again by Theorem~\ref{thm:ioscc},
the following are input-output equations:
 \begin{align} \label{eq:i-o-for-M-tilde}
 	\det (\partial I - \widetilde A) y_i ~=~ (-1)^{i+j} \sum_{j \in In} \det \left( \partial I- \widetilde A \right)_{ji} u_j 
	\quad \quad {\rm for~} i \in Out~.
 \end{align}

Let $c: \mathbb{R}^{|E|} \to \mathbb{R}^k$ be the coefficient map for $\mathcal M$ arising from the coefficients of the input-output equations in~\eqref{eq:i-o-for-M},
where the first coefficient is chosen to be the coefficient of $y_i$ in the left-hand side of \eqref{eq:i-o-for-M}. Notice that this coefficient is independent of the choice of $i \in Out$ and, in fact, is equal to 0 because $\mathcal M$ is strongly connected and has no leaks~\cite{MeshkatSullivantEisenberg}.  

Similarly, let $\widetilde c: \mathbb{R}^{|E|+1} \to \mathbb{R}^k$ denote the coefficient map for $\widetilde{\mathcal{M}}$ coming from the input-output equations in~\eqref{eq:i-o-for-M-tilde}, where the coefficients are chosen in the same order as for $c$.  
The first coefficient,  $\widetilde c_1$,
in contrast with the one in the previous coefficient map, 
is not equal to~0.
In fact, we claim that this coefficient is as follows:
	\begin{align} \label{eq:first-coef}
	\widetilde c_1 ~=~ a_{01} 
				\sum_{\mathcal{T} \in \tau_1} \pi_{\mathcal{T}}~, 
	\end{align}
where $\tau_1$ denotes the set of all (directed) spanning trees of $H$ consisting of $(n-1)$ edges, each of which is directed toward node 1 (the root of the tree), and $\pi_{\mathcal{T}}$ denotes the following monomial in $\mathbb{Q}[a_{ji} \mid (i,j) \in E]$:
	\[
	\pi_{\mathcal{T}} ~:=~ \prod_{(i,j) {\rm~is~an~edge~of~} {\mathcal{T}}} a_{ji}~.
	\]
Indeed, it is straightforward to check equality~\eqref{eq:first-coef} using \cite[Proposition~4.6]{singularlocus}.  

The sum $\sum_{{\mathcal{T}} \in \tau_1} \pi_{\mathcal{T}}$ does not involve $a_{01}$ (by construction) and is a nonzero polynomial (because $H$ is a strongly connected digraph on $n$ nodes and hence contains at least one $(n-1)$-edge spanning tree rooted at node 1).  
Thus, using~\eqref{eq:first-coef}, we see that 
	$\frac{\partial \widetilde c_1}{\partial a_{01}} = \sum_{{\mathcal{T}} \in \tau_1} \pi_{\mathcal{T}}$. 
	So, $\frac{\partial \widetilde c_1}{\partial a_{01}} $ is a nonzero polynomial that does not involve $a_{01}$:
	\begin{align} \label{eq:no-a01}
	 \left( \frac{\partial \widetilde c_1 }{\partial a_{01}} \right) |_{a_{01}=0} ~=~ 
	 	\frac{\partial \widetilde c_1 }{\partial a_{01}} ~=~
		 \sum_{{\mathcal{T}} \in \tau_1} \pi_{\mathcal{T}} \quad \textrm{is a nonzero polynomial.}
	\end{align}


Going beyond the first coordinate, we claim that the full coefficient map has the form:
	\begin{align} \label{eq:c-and-c-tilde}
	\widetilde c ~=~ c + a_{01}v~,
	\end{align}
for some $v \in (\mathbb{Q}[a_{ji} \mid (i,j) \in E])^k$.
Indeed, this claim follows from~\eqref{eq:i-o-for-M} and~\eqref{eq:i-o-for-M-tilde}, and from the fact that
$\widetilde A$ comes from adding $-a_{01}$ to one entry of $A$.  

Now we consider ${\rm Jac}(c)$, the Jacobian matrix of $c$.  
The first row corresponds to the coefficient that we saw is 0, so this row is a row of 0's:
 \begin{align} \label{eq:jac-c}
 	{\rm Jac}(c) ~=~ 
	\begin{pmatrix}
	0 & \dots & 0 \\
	*  & \dots & * \\
	\vdots & \ddots & \vdots \\
	* &  \dots & *
	\end{pmatrix}~.
 \end{align}
When evaluated at a generic point, ${\rm Jac}(c)$ has (full) rank
equal to $|E|$ (because $\mathcal M$ is generically locally
identifiable from $c$, and by Proposition~\ref{prop:jacobian}).
So, there exists a choice of $|E|$ rows of ${\rm Jac}(c)$, which we index by $i_1,\dots, i_{|E|}$,
 so that the resulting
$|E| \times |E|$ 
 matrix, denoted by $J$, is generically full rank.  That is, $\det J$ is a nonzero polynomial in $\mathbb{Q}[a_{ji} \mid (i,j) \in E]$.

By~\eqref{eq:c-and-c-tilde}, we know  
${\rm Jac}(\widetilde c)$ is obtained from 
${\rm Jac}( c)$ by first adding polynomial-multiples of $a_{01}$ to some entries, and then appending a column corresponding
to the new parameter, $a_{01}$:
\begin{align} \label{eq:jac-c-tilde}
{\rm Jac}(\widetilde c) ~=~
	\left( \begin{array}{c|c}
	~ 					& \frac{\partial \widetilde c_1 }{\partial a_{01}} \\
	{\rm Jac}(c)+a_{01}K 	& * \\
	~					& \vdots \\ 
	~					& * \\
 	\end{array} \right)~,
 \end{align} 
for some $k \times |E|$ matrix 
 $K$  
  with entries in $\mathbb{Q}[a_{ji} \mid (i,j) \in E]$.  

Let $\widetilde J$ denote the 
$(|E|+1) \times (|E|+1)$ 
submatrix of ${\rm Jac}(\widetilde c)$ coming from choosing 
the first row and 
 the rows indexed by $i_1,\dots, i_{|E|}$.  
Then, by construction and from 
using equations~\eqref{eq:no-a01}, 
\eqref{eq:jac-c}, and~\eqref{eq:jac-c-tilde}, 
we obtain:
	\begin{align} \label{eq:det-J-tilde}
	\left(\det \widetilde J \right)|_{a_{01}=0} 
	~&=~
	\det \left( \widetilde J|_{a_{01}=0}\right)  
	&=~ 
	\det 
	 \left( \begin{array}{ccc|c}
	0 & \dots & 0 & \frac{\partial  \widetilde c_1 }{\partial a_{01}} \\
\hline
	~& ~& ~ & * \\
	~& J & ~ & \vdots \\
	~& ~& ~ & * \\
 \end{array} \right)
	~&=~
		\pm 	
	 \frac{\partial  \widetilde c_1  }{\partial a_{01}} ~\det J~.
	\end{align}

As noted earlier, $\frac{ \partial  \widetilde c_1  }{\partial a_{01}}$ and $\det J$ are both nonzero polynomials.  So, 
equation~\eqref{eq:det-J-tilde} implies that $\det \widetilde J$ also
is a nonzero polynomial.  So, ${\rm Jac}( \widetilde c)$ generically
has (full) rank equal to $|E|+1$.  And thus, 
by definition and
 Proposition~\ref{prop:jacobian}, the model 
$\widetilde{\mathcal{M}}$ is generically locally identifiable
from the coefficient map.
\end{proof}

\begin{rmk} \label{rmk:similar-pf-s-locus}
Our proof of Theorem~\ref{thm:add-leak} is similar to that for~\cite[Theorem 3.11]{singularlocus}, which analyzed whether identifiability is preserved when an edge of a linear compartmental model is deleted.  Indeed, our results explain why, 
for the models we considered in~\cite{singularlocus}, 
the leak parameters (labeled $a_{01}$ in~\cite{singularlocus}) do not divide the singular-locus equation (see Proposition~\ref{prop:removeleaks}). 
\end{rmk}

We conjecture that the converse of Theorem~\ref{thm:add-leak} holds.

\begin{conjecture} [Deleting one leak] \label{conj:iff-leak}
 Let $\widetilde{\mathcal M}$ be a linear compartmental model that is
 strongly connected and has at least one input and exactly one
 leak.  If $\widetilde{\mathcal M}$ is generically locally
 identifiable from the coefficient map, then so is the model $\mathcal M$ obtained from
 $\widetilde{\mathcal M}$ by removing the leak.
\end{conjecture}

The next result resolves one case 
of Conjecture~\ref{conj:iff-leak}.

\begin{proposition} 
 \label{prop:removeleaks} 
	Let $\mathcal M$ be a linear compartmental model that is strongly connected, and has an input, output, and leak in a single compartment (and has no other inputs, outputs, or leaks).  If $\mathcal M$ is generically locally identifiable from the coefficient map, then so is the model obtained from $\mathcal M$ by removing the leak.
\end{proposition}

\begin{proof}
Assume that $\mathcal M=(G,\{1\},\{1\},\{1\})$ is a generically
locally identifiable (from the coefficient map), strongly connected linear compartmental model with $n$ compartments and an input, output, and leak in  
compartment-1 (and no other inputs, outputs, or leaks).  Let $\widetilde{\mathcal{M}}$ be the model obtained from $\mathcal M$ by removing the leak.

The input-output equation of $\mathcal M$ 
arising from Theorem~\ref{thm:ioscc}
has the following form:
\[
	 y_1^{(n)} + c_{n-1}  y_1^{(n-1)} + \dots  + c_1 y_1' + c_0y_1 
	 ~=~ 
	  u_1^{(n-1)} + d_{n-2} u_1^{(n-2)} + \dots  + d_1 u_1' + d_0 u_1~.
\]
The compartmental matrix for $\mathcal M$, denoted by $A$, comes from adding $-a_{01}$ to the $(1,1)$-entry of $\widetilde A$, the compartmental matrix for 
$\widetilde{\mathcal{M}}$.  Therefore, the matrices $(\partial I- A)_{11}$ and $(\partial I- \widetilde A)_{11}$ are equal. 
So, by Theorem \ref{thm:ioscc}, 
the right-hand side of the input-output equation for $\widetilde{\mathcal M}$ coincides with that for $\mathcal M$ (and these coefficients in common do not involve the leak parameter $a_{01}$).  We therefore write  
the input-output equation for $\widetilde {\mathcal M}$ as follows:
\[
	 y_1^{(n)} + \widetilde{c}_{n-1}  y_1^{(n-1)} + \dots  + \widetilde{c}_1 y_1' + 0 
	 ~=~ 
	  u_1^{(n-1)} + d_{n-2} u_1^{(n-2)} + \dots  + d_1 u_1' + d_0 u_1~.
\]
The coefficient of $y_1$ is 0, 
because this coefficient is the constant term of  $\det (\partial I - A)$, which is $ \pm \det A$, and this determinant is 0 (the column sums of $A$ are zero).

The resulting coefficient map for $\mathcal M$ is:
\[
	{\bf c}_{\mathcal M} :=
	(c_{n-1},~\dots,~ c_1,~ c_0,~ d_{n-2},~ \dots,~ d_1,~ d_0):
		\mathbb{R}^{|E|+1} \to \mathbb{R}^{2n-1}~,
\]
where $E$ is the edge set of $G$; 
and the coefficient map for $\widetilde{\mathcal M}$ is:
\[
	{\bf c}_{\widetilde{\mathcal M}} :=
	(\widetilde{c}_{n-1},~\dots,~ \widetilde{c}_1,~ d_{n-2},~ \dots,~ d_1,~ d_0):
		\mathbb{R}^{|E|} \to \mathbb{R}^{2n-2}~.
\]

From \cite[Theorem 4.5]{singularlocus}, the coefficients $c_i$ and $d_i$ 
are sums of products of edge labels, where the sum is taken over all $(n-i)$-edge (respectively, $(n-i-1)$-edge) spanning incoming forests 
 in the ``leak-augmented graph'' of $G$ (respectively, a related graph obtained by deleting compartment-1).  The coefficients $\widetilde{c}_i$ have a similar interpretation, where 
 now the leak-augmented graph is simply $G$, as $\widetilde {\mathcal M}$ has no leaks.  It is straightforward to check, then, that the following equalities hold:
 \begin{align} \label{eq:reln-coef}
 \notag
 	\widetilde{c}_{n-1} ~&=~ c_{n-1} -a_{01} \\ 
\notag
 	\widetilde{c}_{n-2} ~&=~ c_{n-2} -a_{01}d_{n-2}\\
	&~~ \vdots\\
\notag
 	\widetilde{c}_1~&=~ c_{1} -a_{01}d_{1}\\
\notag
 	0		~&=~ c_{0} -a_{01}d_{0}~.
 \end{align}
We claim that ${\bf c}_{\mathcal M}$, the coefficient map for $\mathcal M$,
is generically finite-to-one if and only if 
the following map is generically finite-to-one: 
\[
	\phi:=(a_{01},~\widetilde{c}_{n-1},~\dots,~ \widetilde{c}_1,~ d_{n-2},~ \dots,~ d_1,~ d_0):
		\mathbb{R}^{|E|+1} \to \mathbb{R}^{2n-1}~.
\]
To verify this claim, we first define the map:
\begin{align*}
	\nu:
		\mathbb{R}^{2n-1} &\to \mathbb{R}^{2n-1}\\
		(a_{01}^*,~\widetilde{c}_{n-1}^*,~\dots,~ \widetilde{c}_1^*,~ d_{n-2}^*,~ \dots,~ d_0^*) & \mapsto
		(\widetilde{c}_{n-1}^*+a_{01}^*,~\widetilde{c}_{n-2}^*+a_{01}^*d_{n-2}^*,~
		~\dots,~ \widetilde{c}_1^*+a_{01}^*d_{1}^*,~
		a_{01}^*d_0^*,~ \\
	& \quad \quad	\quad	 d_{n-2}^*,~ \dots,~ d_0^*) ~.
\end{align*}
By the equations~\eqref{eq:reln-coef}, we have that $\nu \circ \phi = {\bf c}_{\mathcal M}$.  Also, it is straightforward to check that ${\rm Jac}~\nu= \pm d_0$, which is a nonzero polynomial, and so this Jacobian matrix is generically full rank.  Our claim now follows directly.


Hence, the Jacobian matrix of $\phi$, displayed below (we order the variables with $a_{01}$  first, so it corresponds to the first column), generically has (full) rank=$2n-1$:
\[
{\rm Jac}~\phi~=~
	 \left( \begin{array}{c|ccc}
	1 & 0 &  \dots & 0  \\
\hline
	0 &~& ~& ~  \\
	 \vdots&~& {\rm Jac}~{\bf c}_{\widetilde{\mathcal M}} & ~  \\
	0&~& ~& ~ \\
 \end{array} \right)~.
 \]
 Therefore, ${\rm Jac}~{\bf c}_{\widetilde{\mathcal M}}$, the
 Jacobian matrix of the coefficient map for $\widetilde{\mathcal M}$, 
 generically has (full) rank=$2n-2$.  So, by definition
   and
Proposition~\ref{prop:jacobian}, 
 the model $\widetilde{\mathcal M}$ is generically locally
 identifiable from the coefficient map. 
\end{proof}

We now apply Proposition~\ref{prop:removeleaks}
to three well-known families of linear compartmental models: catenary (path graph) models, mammillary (star graph) models, and cycle models~\cite{godfrey}.  
See Figures~\ref{fig:cat} and~\ref{fig:cycle}.
For these models, adding a leak to compartment-1 yields a model that is generically locally identifiable~\cite{cobelli-constrained-1979, singularlocus, MeshkatSullivant, MeshkatSullivantEisenberg}.  So, Proposition~\ref{prop:removeleaks} yields the following result:

\begin{proposition} \label{prop:3-models}
These linear compartmental models are generically locally identifiable
from the coefficient map:
\begin{enumerate}
\item the $n$-compartment catenary (path) model in Figure~\ref{fig:cat} (for $n \geq 2$),
\item the $n$-compartment cycle model in Figure~\ref{fig:cycle} (for $n \geq 3$), and
\item the $n$-compartment mammillary (star) model in Figure~\ref{fig:cycle} (for $n \geq 2$).
\end{enumerate}
\end{proposition}

\begin{rmk} \label{rmk:cobelli}
Parts (1) and (3) of Proposition~\ref{prop:3-models} were previously
known~\cite[\S 3.1 and 4.1]{cobelli-constrained-1979}, while, to our knowledge, part (2) is new.  
The advantage of our approach over that in~\cite{cobelli-constrained-1979} is that we were able to avoid computations: Proposition~\ref{prop:removeleaks} allows us to immediately ``transfer'' a prior result to a related family of models.
\end{rmk}

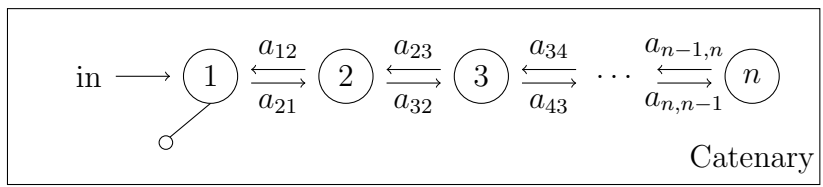
\begin{figure}[ht]
\begin{center}
	\begin{tikzpicture}[scale=1.8]
 	\draw (0,0) circle (0.2);	
 	\draw (1,0) circle (0.2);	
 	\draw (2,0) circle (0.2);	
 	\draw (4,0) circle (0.2);	
    	\node[] at (0, 0) {1};
    	\node[] at (1, 0) {$2$};
    	\node[] at (2, 0) {$3$};
    	\node[] at (3, 0) {$\dots$};
    	\node[] at (4, 0) {$n$};
	 \draw[<-] (0.3, 0.05) -- (0.7, 0.05);	
	 \draw[->] (0.3, -0.05) -- (0.7, -0.05);	
	 \draw[<-] (1.3, 0.05) -- (1.7, 0.05);	
	 \draw[->] (1.3, -0.05) -- (1.7, -0.05);	
	 \draw[<-] (2.3, 0.05) -- (2.7, 0.05);	
	 \draw[->] (2.3, -0.05) -- (2.7, -0.05);	
	 \draw[<-] (3.3, 0.05) -- (3.7, 0.05);	
	 \draw[->] (3.3, -0.05) -- (3.7, -0.05);	
   	 \node[] at (0.5, 0.2) {$a_{12}$};
   	 \node[] at (1.5, 0.2) {$a_{23}$};
   	 \node[] at (2.5, 0.2) {$a_{34}$};
   	 \node[] at (3.5, 0.2) {$a_{n-1,n}$};
   	 \node[] at (0.5, -0.2) {$a_{21}$};
   	 \node[] at (1.5, -0.2) {$a_{32}$};
   	 \node[] at (2.5, -0.2) {$a_{43}$};
   	 \node[] at (3.5, -0.2) {$a_{n,n-1}$};
 	\draw (-0.33,-.49) circle (0.05);	
	 \draw[-] (0, -.2 ) -- (-0.3, -.45);	
	 \draw[->] (-0.7, 0) -- (-0.3, 0);	
   	 \node[] at (-.9, 0) {in};
    	\node[above] at (4, -0.8) {Catenary};
\draw (-1.5,-.8) rectangle (4.5, .5);
	\end{tikzpicture}
\end{center}
\caption{The {\bf catenary} (path) model with $n$ compartments and no leaks, in which compartment-1 has an input and output (cf.~\cite[Figure 1]{singularlocus}).
}
\label{fig:cat}
\end{figure}

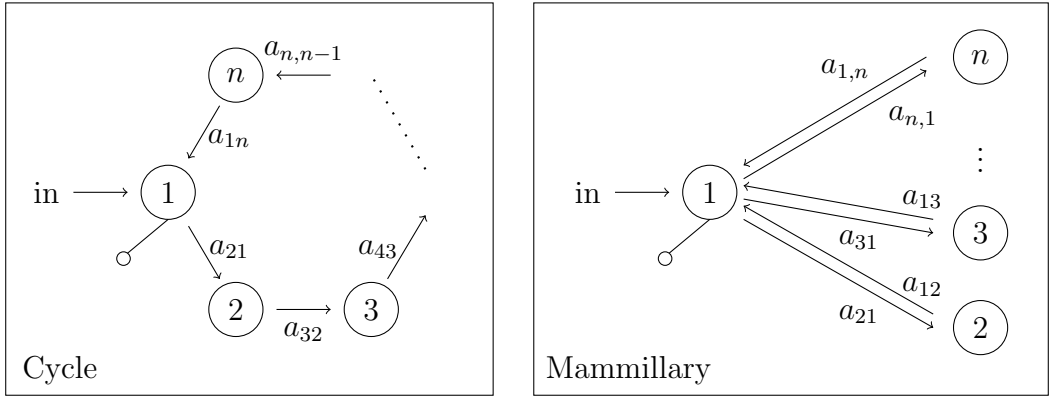
\begin{figure}[ht]
\begin{center}
	\begin{tikzpicture}[scale=1.8]
 	\draw (-1,0) circle (0.2);	
 	\draw (-0.5,0.866) circle (0.2);	
 	\draw (0.5,-0.866) circle (0.2);	
 	\draw (-0.5,-0.866) circle (0.2);	
	 \draw[->] (-0.85, -0.25) -- (-0.62, -0.64 );
 	 \draw[->]  (-0.2,-0.866) -- (0.2,-0.866) ;
	 \draw[->]  (0.62, -0.64) -- (0.9,-0.17); 
	 \draw[loosely dotted,thick] (0.9,0.17) -- (0.5,0.866);
	 \draw[->]  (0.2,0.866) -- (-0.2,0.866);
	 \draw[->] (-0.62, 0.64 ) -- (-0.85, 0.25)  ;
   	 \node[] at (0.0, -1.03) {$a_{32}$};
   	 \node[] at (0.0, 1.03) {$a_{n,n-1}$};
   	 \node[] at (-0.55, -0.42) {$a_{21}$};
   	 \node[] at (-0.55, 0.4) {$a_{1n}$};
   	 \node[] at (0.55, -0.42) {$a_{43}$};
    	\node[] at (-1, 0) {1};
    	\node[] at  (-0.5,-0.866) {$2$};
    	\node[] at  (+0.5,-0.866) {$3$};
    	\node[] at (-0.5,0.866) {$n$};
 	\draw (-1.33,-.49) circle (0.05);	
	 \draw[-] (-1, -.2 ) -- (-1.3, -.45);	
	 \draw[->] (-1.7, 0) -- (-1.3, 0);	
   	 \node[] at (-1.9, 0) {in};
    	\node[above] at (-1.8, -1.5) {Cycle};
\draw (-2.2,-1.5) rectangle (1.4, 1.4);
 	\draw (3,0) circle (0.2);
 	\draw (5,-1) circle (0.2);
 	\draw (5,-0.3) circle (0.2);
 	\draw (5,1) circle (0.2);
    	\node[] at (3, 0) {1};
    	\node[] at (5, -1) {2};
    	\node[] at (5, -0.3) {3};
    	\node[] at (5, 0.3) {$\vdots$};
    	\node[] at (5, 1) {$n$};
	 \draw[->] (3.25, -.2) -- (4.65, -1);
	 \draw[<-] (3.25, -.1) -- (4.65, -0.9);
	 \draw[->] (3.25, -0.05) -- (4.65, -0.3);
	 \draw[<-] (3.25, 0.05) -- (4.65, -0.2);
	 \draw[->] (3.25, 0.1) -- (4.6, 0.9);
	 \draw[<-] (3.25, 0.2) -- (4.6, 1);
   	 \node[] at (4.1, -0.9) {$a_{21}$};
   	 \node[] at (4.57, -0.7) {$a_{12}$};
   	 \node[] at (4.57, -0.05) {$a_{13}$};
   	 \node[] at (4.1, -0.35) {$a_{31}$};
   	 \node[] at (4, 0.9) {$a_{1,n}$};
   	 \node[] at (4.5, 0.55) {$a_{n,1}$};
 	\draw (2.67,-.49) circle (0.05);	
	 \draw[-] (3, -.2 ) -- (2.7, -.45);	
	 \draw[->] (2.3, 0) -- (2.7, 0);	
   	 \node[] at (2.1, 0) {in};
\draw (1.7,-1.5) rectangle (5.5, 1.4);
    	\node[above] at (2.4, -1.5) {Mammillary};
	\end{tikzpicture}
\end{center}
\caption{Two models with $n$ compartments and no leaks, where compartment-1 has an input and output (cf.~\cite[Figure 2]{singularlocus}).
{ Left:} The {\bf cycle}.  { Right:} The {\bf mammillary} (star). 
}
\label{fig:cycle}
\end{figure}

\section{Examples of adding or removing inputs, outputs, leaks, or edges} \label{sec:ex}
This section compiles examples of linear compartmental models that show that 
when a leak or edge is added -- or when an input, output, leak, or edge is deleted -- identifiability is sometimes preserved and sometimes lost (as summarized in Table~\ref{table:add-delete}).
Our examples also show, conversely, that 
when a leak or edge is deleted -- or when an input, output, leak, or edge is added -- 
unidentifiability is sometimes preserved and sometimes lost (Table~\ref{table:add-delete}).

In the following examples, 
each input-output equation we present comes from a
  characteristic set (for the model under consideration). 
This fact can be checked, by hand or using {\tt DAISY}~\cite{daisy},

\begin{eg}[Add or delete a leak] \label{ex:add-delete-leak}
We consider models $\mathcal M$ and $\widetilde{\mathcal M}$, where 
$\widetilde{\mathcal M}$ is obtained from $\mathcal M$ by adding a leak (or, equivalently, 
$\mathcal M$ is obtained from $\widetilde{\mathcal M}$ by deleting a leak).

By Theorem~\ref{thm:add-leak}, 
when $\mathcal M$ is strongly connected and has inputs but no leaks, then if $\mathcal M$ is identifiable, then $\widetilde{\mathcal M}$ is too.

Next, let $\mathcal M$,  $\widetilde{\mathcal M}$, and 
$\widetilde{\mathcal M}'$ 
denote the models on the left, middle, and right, respectively:
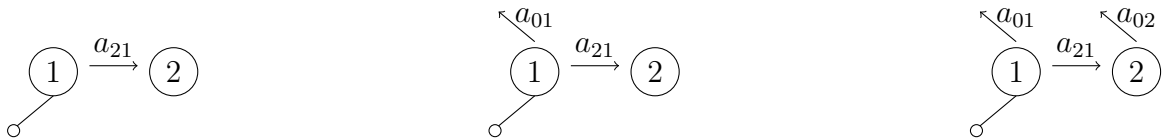
\begin{center}
	\begin{tikzpicture}[scale=1.6]
 	\draw (0,0) circle (0.2);	
 	\draw (1,0) circle (0.2);	
    	\node[] at (0, 0) {1};
    	\node[] at (1, 0) {$2$};
	 \draw[->] (0.3, 0.05) -- (0.7, 0.05);	
   	 \node[] at (0.5, 0.2) {$a_{21}$};
 	\draw (-0.33,-.49) circle (0.05);	
	 \draw[-] (0, -.2 ) -- (-0.3, -.45);	
 	\draw (4,0) circle (0.2);	
 	\draw (5,0) circle (0.2);	
    	\node[] at (4, 0) {1};
    	\node[] at (5, 0) {$2$};
	 \draw[->] (4.3, 0.05) -- (4.7, 0.05);	
   	 \node[] at (4.5, 0.2) {$a_{21}$};
 	\draw (3.67,-.49) circle (0.05);	
	 \draw[-] (4, -.2 ) -- (3.7, -.45);	
	 \draw[->] (4, .25) -- (3.7, .5);	
   	 \node[] at (4, 0.45) {$a_{01}$};
 	\draw (8,0) circle (0.2);	
 	\draw (9,0) circle (0.2);	
    	\node[] at (8, 0) {1};
    	\node[] at (9, 0) {$2$};
	 \draw[->] (8.3, 0.05) -- (8.7, 0.05);	
   	 \node[] at (8.5, 0.2) {$a_{21}$};
 	\draw (7.67,-.49) circle (0.05);	
	 \draw[-] (8, -.2 ) -- (7.7, -.45);	
	 \draw[->] (8, .25) -- (7.7, .5);	
   	 \node[] at (8, 0.45) {$a_{01}$};	
	 \draw[->] (9, .25) -- (8.7, .5);	
   	 \node[] at (9, 0.45) {$a_{02}$};	
	 \end{tikzpicture}
\end{center}
The model  $\mathcal M$ is identifiable (the input-output equation is $ y_1' + a_{21} y_1 = 0$),
while $\widetilde{\mathcal M}$ is not (the input-output equation is $ y_1' + (a_{01}+a_{21}) y_1 = 0$), nor is $\widetilde{\mathcal M}'$ (recall Example~\ref{ex:uniden}).
\end{eg}

Example~\ref{ex:add-delete-leak} does not give an example of 
an unidentifiable model $\mathcal M$
and an identifiable model $\widetilde{\mathcal M}$, where $\widetilde{\mathcal M}$ has one more leak than $\mathcal M$.  We do not know whether such an example exists, so we pose the following question.

\begin{question}\label{q:leak}
For an unidentifiable linear compartmental model $\mathcal M$, if one leak is added, is the resulting model always unidentifiable?
\end{question}

\begin{eg}[Add or delete an edge] \label{ex:add-delete-edge}
We consider models $\mathcal M$ and $\widetilde{\mathcal M}$, where 
$\widetilde{\mathcal M}$ comes from 
adding an edge to $\mathcal M$ (or, equivalently, 
$\mathcal M$ comes from deleting an edge of $\widetilde{\mathcal M}$).

In prior work, we showed that 
when both models are strongly connected, and the 
parameter of the 
deleted edge 
does not divide the ``singular-locus equation'' of 
$\widetilde{\mathcal M}$, then deleting the edge 
preserves generic local identifiability (see \cite[Example 3.2]{singularlocus}).  

On the other hand, 
the model shown in 
 the right-hand side of~\eqref{eq:model} 
in Example~\ref{ex:delete-out} 
below is identifiable~\cite{singularlocus}, but deleting the edge labeled $a_{12}$ yields an unidentifiable model.

Next, let $\mathcal M$ denote the model on the left, and $\widetilde{\mathcal M}$ the model on the right:
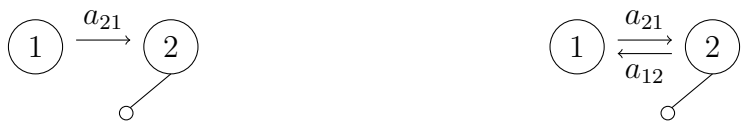
\begin{center}
	\begin{tikzpicture}[scale=1.8]
 	\draw (0,0) circle (0.2);	
 	\draw (1,0) circle (0.2);	
    	\node[] at (0, 0) {1};
    	\node[] at (1, 0) {$2$};
	 \draw[->] (0.3, 0.05) -- (0.7, 0.05);	
   	 \node[] at (0.5, 0.2) {$a_{21}$};
 	\draw (0.67,-.49) circle (0.05);	
	 \draw[-] (1, -.2 ) -- (.7, -.45);	
 	\draw (4,0) circle (0.2);	
 	\draw (5,0) circle (0.2);	
    	\node[] at (4, 0) {1};
    	\node[] at (5, 0) {$2$};
	 \draw[->] (4.3, 0.05) -- (4.7, 0.05);	
	 \draw[<-] (4.3, -0.05) -- (4.7, -0.05);	
   	 \node[] at (4.5, 0.2) {$a_{21}$};
   	 \node[] at (4.5, -0.2) {$a_{12}$};
 	\draw (4.67,-.49) circle (0.05);	
	 \draw[-] (5, -.2 ) -- (4.7, -.45);	
	\end{tikzpicture}
\end{center}
The model  $\mathcal M$ is identifiable (the input-output equation is $y_2'' + a_{21}  y_2' = 0$),
while $\widetilde{\mathcal M}$ is not (the input-output equation is $y_2'' + (a_{12}+a_{21})  y_2' = 0$).

Finally, if $\mathcal{M}$ and $\widetilde{\mathcal M}$ have no outputs, or have more edges and leaks than the number of coefficients in the input-output equations, then both models will be unidentifiable.  For instance, let $\mathcal M$ denote the model on the left, and $\widetilde{\mathcal M}$ the model on the right:
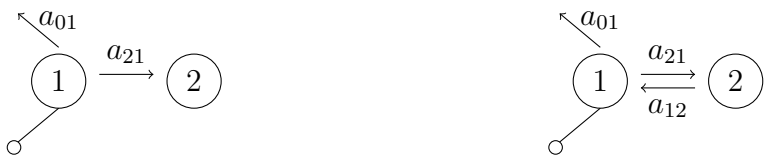
\begin{center}
	\begin{tikzpicture}[scale=1.8]
 	\draw (0,0) circle (0.2);	
 	\draw (1,0) circle (0.2);	
    	\node[] at (0, 0) {1};
    	\node[] at (1, 0) {$2$};
	 \draw[->] (0.3, 0.05) -- (0.7, 0.05);	
   	 \node[] at (0.5, 0.2) {$a_{21}$};
 	\draw (-0.33,-.49) circle (0.05);	
	 \draw[-] (0, -.2 ) -- (-0.3, -.45);	
	 \draw[->] (0, .25) -- (-0.3, .5);	
   	 \node[] at (0, 0.45) {$a_{01}$};
 	\draw (4,0) circle (0.2);	
 	\draw (5,0) circle (0.2);	
    	\node[] at (4, 0) {1};
    	\node[] at (5, 0) {$2$};
	 \draw[->] (4.3, 0.05) -- (4.7, 0.05);	
	 \draw[<-] (4.3, -0.05) -- (4.7, -0.05);	
   	 \node[] at (4.5, 0.2) {$a_{21}$};
   	 \node[] at (4.5, -0.2) {$a_{12}$};
 	\draw (3.67,-.49) circle (0.05);	
	 \draw[-] (4, -.2 ) -- (3.7, -.45);	
	 \draw[->] (4, .25) -- (3.7, .5);	
   	 \node[] at (4, 0.45) {$a_{01}$};
	\end{tikzpicture}
\end{center}
Both models are unidentifiable 
($\mathcal M$ was analyzed in Example~\ref{ex:add-delete-leak}, 
	and 
	the input-output equation for 
 $\widetilde{\mathcal M}$ is 
 	$y_1'' + 
	(a_{01}+a_{12}+a_{21})  y_1' + (a_{01}a_{12})y_1 = 0$).

\end{eg}

\begin{eg}[Add or delete an input] \label{ex:delete-in}

We consider models $\mathcal M$ and $\widetilde{\mathcal M}$, where 
$\widetilde{\mathcal M}$ comes from 
adding an input to $\mathcal M$ (or, equivalently, 
$\mathcal M$ comes from deleting an input of $\widetilde{\mathcal M}$).
If $\mathcal M$ is identifiable, then $\widetilde{\mathcal M}$ is too (by Proposition~\ref{prop:add-in-out}, assuming there is at least one input) and so $\widetilde{\mathcal M}$ can not be unidentifiable. Here we show the remaining three combinations of identifiability/unidentifiability occur.

The following linear compartmental model is globally identifiable, and so is the model obtained by removing the input:

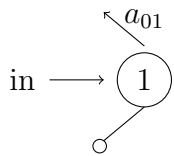
\begin{center}
	\begin{tikzpicture}[scale=1.8]
 	\draw (0,0) circle (0.2);	
    	\node[] at (0, 0) {1};
 	\draw (-0.33,-.49) circle (0.05);	
	 \draw[-] (0, -.2 ) -- (-0.3, -.45);	
	 \draw[->] (-0.7, 0) -- (-0.3, 0);	
   	 \node[] at (-.9, 0) {in};
	 \draw[->] (0, .25) -- (-0.3, .5);	
   	 \node[] at (0, 0.45) {$a_{01}$};
	\end{tikzpicture}
\end{center}	

On the other hand, the model displayed on the right-hand side of~\eqref{eq:model} 
in Example~\ref{ex:delete-out} 
below is identifiable~\cite{singularlocus},
but removing the input yields an unidentifiable model (as we saw in Example \ref{ex:add-delete-edge}, the input-output equation is 
	$y_1'' + (a_{01}+a_{12}+a_{21}) y_1' + (a_{01}a_{12}) y_1 = 0$).

Finally, let $\mathcal M$ denote the model on the left, and $\widetilde{\mathcal M}$ the model on the right:
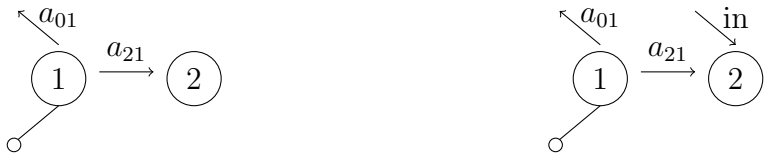
\begin{center}
	\begin{tikzpicture}[scale=1.8]
 	\draw (4,0) circle (0.2);	
 	\draw (5,0) circle (0.2);	
    	\node[] at (4, 0) {1};
    	\node[] at (5, 0) {$2$};
	 \draw[->] (4.3, 0.05) -- (4.7, 0.05);	
   	 \node[] at (4.5, 0.2) {$a_{21}$};
 	\draw (3.67,-.49) circle (0.05);	
	 \draw[-] (4, -.2 ) -- (3.7, -.45);	
	 \draw[->] (4, .25) -- (3.7, .5);	
   	 \node[] at (4, 0.45) {$a_{01}$};
 	\draw (8,0) circle (0.2);	
 	\draw (9,0) circle (0.2);	
    	\node[] at (8, 0) {1};
    	\node[] at (9, 0) {$2$};
	 \draw[->] (8.3, 0.05) -- (8.7, 0.05);	
   	 \node[] at (8.5, 0.2) {$a_{21}$};
 	\draw (7.67,-.49) circle (0.05);	
	 \draw[-] (8, -.2 ) -- (7.7, -.45);	
	 \draw[->] (8, .25) -- (7.7, .5);	
   	 \node[] at (8, 0.45) {$a_{01}$};
	 \draw[<-] (9, .25) -- (8.7, .5);	
   	 \node[] at (9, 0.45) {in};	
\end{tikzpicture}
\end{center}
Both models are unidentifiable ($\mathcal M$ was analyzed in Example~\ref{ex:add-delete-leak}, and
by Theorem~\ref{thm:ioscc}
 adding an input ``downstream'' from the output does not affect 
the input-output equation).
\end{eg}

\begin{eg}[Add or delete an output] \label{ex:delete-out}
We consider models $\mathcal M$ and $\widetilde{\mathcal M}$, where 
$\widetilde{\mathcal M}$ comes from 
adding an output to $\mathcal M$ (or, equivalently, 
$\mathcal M$ comes from deleting an output of $\widetilde{\mathcal M}$).
It is not possible for $\mathcal M$ to be identifiable while $\widetilde{\mathcal M}$ is not (by Proposition~\ref{prop:add-in-out}, assuming at least one input). Here we show the remaining three combinations of identifiability/unidentifiability.

The model on the right is generically locally identifiable~\cite{singularlocus},
and thus, by Proposition~\ref{prop:add-in-out}, so is the one on the left, which is obtained by adding an output:
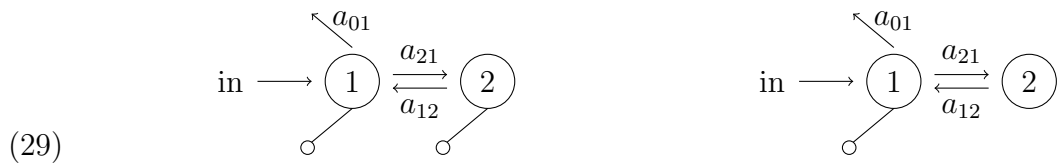
\begin{center}
\begin{align} \label{eq:model}
	\begin{tikzpicture}[scale=1.8]
 	\draw (0,0) circle (0.2);	
 	\draw (1,0) circle (0.2);	
    	\node[] at (0, 0) {1};
    	\node[] at (1, 0) {$2$};
	 \draw[->] (0.3, 0.05) -- (0.7, 0.05);	
	 \draw[<-] (0.3, -0.05) -- (0.7, -0.05);	
   	 \node[] at (0.5, 0.2) {$a_{21}$};
   	 \node[] at (0.5, -0.2) {$a_{12}$};
 	\draw (-0.33,-.49) circle (0.05);	
	 \draw[-] (0, -.2 ) -- (-0.3, -.45);	
 	\draw (0.67,-.49) circle (0.05);	
	 \draw[-] (1, -.2 ) -- (.7, -.45);	
	 \draw[->] (-0.7, 0) -- (-0.3, 0);	
   	 \node[] at (-.9, 0) {in};
	 \draw[->] (0, .25) -- (-0.3, .5);	
   	 \node[] at (0, 0.45) {$a_{01}$};
 	\draw (4,0) circle (0.2);	
 	\draw (5,0) circle (0.2);	
    	\node[] at (4, 0) {1};
    	\node[] at (5, 0) {$2$};
	 \draw[->] (4.3, 0.05) -- (4.7, 0.05);	
	 \draw[<-] (4.3, -0.05) -- (4.7, -0.05);	
   	 \node[] at (4.5, 0.2) {$a_{21}$};
   	 \node[] at (4.5, -0.2) {$a_{12}$};
 	\draw (3.67,-.49) circle (0.05);	
	 \draw[-] (4, -.2 ) -- (3.7, -.45);	
	 \draw[->] (3.3, 0) -- (3.7, 0);	
   	 \node[] at (3.1, 0) {in};
	 \draw[->] (4, .25) -- (3.7, .5);	
   	 \node[] at (4, 0.45) {$a_{01}$};
	\end{tikzpicture}
\end{align}
\end{center}

On the other hand, any identifiable model with only one output (such as the one 
depicted in Example~\ref{ex:delete-in}) becomes unidentifiable when the output is deleted.

Finally, let $\mathcal M$ denote the model on the left, and $\widetilde{\mathcal M}$ the model on the right:
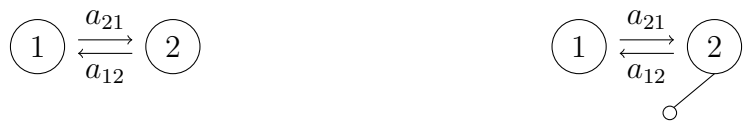
\begin{center}
	\begin{tikzpicture}[scale=1.8]
 	\draw (0,0) circle (0.2);	
 	\draw (1,0) circle (0.2);	
    	\node[] at (0, 0) {1};
    	\node[] at (1, 0) {$2$};
	 \draw[->] (0.3, 0.05) -- (0.7, 0.05);	
	 \draw[<-] (0.3, -0.05) -- (0.7, -0.05);	
   	 \node[] at (0.5, 0.2) {$a_{21}$};
   	 \node[] at (0.5, -0.2) {$a_{12}$};
 	\draw (4,0) circle (0.2);	
 	\draw (5,0) circle (0.2);	
    	\node[] at (4, 0) {1};
    	\node[] at (5, 0) {$2$};
	 \draw[->] (4.3, 0.05) -- (4.7, 0.05);	
	 \draw[<-] (4.3, -0.05) -- (4.7, -0.05);	
   	 \node[] at (4.5, 0.2) {$a_{21}$};
   	 \node[] at (4.5, -0.2) {$a_{12}$};
 	\draw (4.67,-.49) circle (0.05);	
	 \draw[-] (5, -.2 ) -- (4.7, -.45);	
	\end{tikzpicture}
\end{center}
Then $\mathcal M$ is unidentifiable (it has no outputs), and so is 
$\widetilde{\mathcal M}$ (as shown in Example~\ref{ex:add-delete-edge}).

\end{eg}

\section{Discussion} \label{sec:discussion}
This work addresses some fundamental questions pertaining to identifiability of linear compartmental models.  Specifically, we proved results clarifying the effect of adding or removing parts of a model.  Along the way, we showed that a model's input-output equations are influenced by the model's output-reachable subgraph.

Our results together form a step toward addressing the important problem of cataloguing identifiable linear compartmental models.  In such a catalogue, it will be enough to include only identifiable models that are minimal with respect to their input and output sets, as those with more inputs or outputs 
are automatically identifiable (Proposition~\ref{prop:add-in-out}).
Another simplification is that certain models with no leaks have the same identifiability properties as those with exactly one leak (Theorem~\ref{thm:add-leak} and Proposition~\ref{prop:removeleaks}).
We expect that our results, and future related results, will shine light on the important problem of determining the precise properties that make models identifiable.

\subsection*{Acknowledgements}
{\small
This project began at a SQuaRE (Structured Quartet Research Ensemble) at AIM, and the authors thank AIM for providing financial support and an excellent working environment.  
EG was supported by the NSF (DMS-1620109).  
HAH gratefully acknowledges funding from EPSRC Postdoctoral Fellowship (EP/K041096/1) and a Royal Society University Research Fellowship.  NM was partially supported by the Clare Boothe Luce Program from the Luce Foundation.  AS was partially supported by the NSF (DMS-1752672)
and the Simons Foundation (\#521874).  The authors thank Gleb Pogudin
and Peter Thompson for bringing the example in Remark~\ref{rmk:gleb}
to our attention, 
Alexey Ovchinnikov for helpful discussions, 
 and  two
conscientious referees for their helpful comments.
}

\bibliographystyle{plain}
\bibliography{square}
 
\end{document}